\documentclass[11pt]{article}

\usepackage[utf8]{inputenc}
\usepackage[T1]{fontenc}
\usepackage{authblk}
\usepackage{graphicx}
\usepackage{grffile}
\usepackage{longtable}
\usepackage{wrapfig}
\usepackage{rotating}
\usepackage[normalem]{ulem}
\usepackage{textcomp}
\usepackage{capt-of}
\usepackage{hyperref}
\usepackage{url}
\usepackage{mathtools}
\usepackage[nolist, nohyperlinks]{acronym}
\usepackage{amsmath}
\usepackage{amssymb}
\usepackage{amsthm}
\usepackage{stmaryrd}
\usepackage{todonotes}
\usepackage{pdflscape}
\usepackage{enumitem}
\usepackage{tikz}
\usetikzlibrary{cd}
\usepackage{ebproof}

\usepackage{fixme}
\fxsetup{layout=footnote}
\fxsetup{status=draft}

\DeclareUrlCommand\emailUrl{\urlstyle{rm}}
\newcommand{\email}[1]{\href{mailto:#1}{\protect\emailUrl{#1}}}
\newcommand{\emailVierling}{\email{jannik.vierling@tuwien.ac.at}}
\newcommand{\emailHetzl}{\email{stefan.hetzl@tuwien.ac.at}}

\newtheorem{theorem}{Theorem}[section]
\newtheorem{proposition}[theorem]{Proposition}
\newtheorem{definition}[theorem]{Definition}
\newtheorem{lemma}[theorem]{Lemma}
\newtheorem{corollary}[theorem]{Corollary}

\newtheorem{remark}[theorem]{Remark}
\newtheorem{example}[theorem]{Example}
\newtheorem{question}[theorem]{Question}

\newcommand{\GSI}[3]{{#1}\text{-}\mathrm{GSI}^{#2}(#3)}
\newcommand{\SI}{\mathrm{SI}}
\newcommand{\IND}[1]{{#1}\text{-}\mathrm{IND}}

\newcommand{\INDParameterFree}[1]{{#1}\text{-}\mathrm{IND}^{-}}

\newcommand{\UINDR}{\mathrm{IND}^{R}}
\newcommand{\INDRPF}[1]{{#1}\text{-}\mathrm{GIND}^{R}}
\newcommand{\LiteralAINDR}{\Literal\text{-}\mathrm{AIND}^{R}}
\newcommand{\DIND}[1]{\IND{#1}_{2}}

\newcommand{\SA}{\mathrm{SA}}
\newcommand{\SASchema}[1]{{#1}\text{-}\SA}
\newcommand{\FV}{\mathrm{FV}}

\newcommand{\Open}{\mathrm{Open}}
\newcommand{\Literal}{\mathrm{Literal}}

\newcommand{\cnf}{\mathit{CNF}}
\newcommand{\sk}{\mathit{sk}}
\newcommand{\skA}{\sk^{\forall}}
\newcommand{\skE}{\sk^{\exists}}

\newcommand{\numeral}[1]{\overline{#1}}

\newcommand{\Aitp}{Automated inductive theorem proving (AITP)}

\newcommand{\QuantifierParenthesis}[3]{({#1}{#2}){#3}}

\newcommand{\Quantifier}[3]{\QuantifierParenthesis{#1}{#2}{#3}}
\newcommand{\Forall}[2]{\Quantifier{\forall}{#1}{#2}}
\newcommand{\ForallBounded}[4]{\Forall{#1{#2}{#3}}{#4}}
\newcommand{\Exists}[2]{\Quantifier{\exists}{#1}{#2}}
\newcommand{\ExistsExOne}[2]{\Quantifier{\exists!}{#1}{#2}}

\newcommand{\DocumentTitle}{
  Induction and Skolemization in saturation theorem proving
}

\author[1,2]{Stefan Hetzl}
\author[1,3]{Jannik Vierling}
\affil[1]{Vienna University of Technology\protect\\Institute of Discrete Mathematics and Geometry}
\affil[2]{\emailHetzl (corresponding author)}
\affil[3]{\emailVierling}
\date{}
\title{\DocumentTitle}

\hypersetup{
 pdfauthor={Jannik Vierling},
 pdftitle={\DocumentTitle},
 pdfkeywords={},
 pdfsubject={},
 pdfcreator={}, 
 pdflang={English}}

\begin{document}

\maketitle

\begin{abstract}
  We consider a typical integration of induction in saturation-based theorem provers and investigate the effects of Skolem symbols occurring in the induction formulas.
  In a practically relevant setting we establish a Skolem-free characterization of refutation in saturation-based proof systems with induction.
  Finally, we use this characterization to obtain unprovability results for a concrete saturation-based induction prover.
\end{abstract}


\section{Introduction}
\label{sec:introduction}

% Automated inductive theorem proving

\Aitp\ is a branch of automated deduction that aims at automating the process of finding proofs that involve mathematical induction.
In first-order \ac{ATP} we try to establish validity whereas in \ac{AITP} one is usually interested to prove that a formula is true in the standard model of some inductive type, such as natural numbers, lists, trees, etc.
By Gödel's incompleteness theorems, truth in the standard model is in general not semi-decidable (even worse, it is in general not even arithmetically definable).
Hence, for \ac{AITP} there is a lot more freedom in the choice of proof systems, than there is for \ac{ATP}.
In practice we see methods that make use of typical first-order induction schemata, Hilbert-style induction rules (for example \cite{kersani2013,kersani2014}), and even more exotic cyclic calculi (see \cite{brotherston2005,brotherston2012}) that can exceed the power of the first-order induction schema \cite{berardi2017a,berardi2019}.

The most prominent applications of automated inductive theorem proving are found in formal methods for software engineering.
For example, the formal verification of software relies strongly on one or another form of induction since any non-trivial program contains some form of loops or recursion.
Besides the applications in software engineering, \ac{AITP} methods have applications in the formalization of mathematics.
For instance, \ac{AITP} methods can be employed by proof assistants to explore a theory in order to provide useful lemmas  \cite{johansson2014}, \cite{johansson2009}.

A wide variety of methods for automated inductive theorem proving have been developed: there are methods based on recursion analysis \cite{moore1979,stevens1988,bundy1989}, proof by consistency \cite{comon2001}, rippling \cite{bundy1993}, cyclic proofs \cite{brotherston2012}, extensions of saturation-based provers \cite{biundo1986,kersani2013,kersani2014,cruanes2015,cruanes2017,echenim2020,reger2019,hajdu2020,wand2017}, tree grammar provers \cite{eberhard2015}, theory exploration based provers \cite{claessen2013a}, rewriting induction \cite{reddy1990}, encoding \cite{schoisswohl2020}, extensions of SMT solvers \cite{reynolds2015}.
Many methods integrate the induction mechanism more or less tightly within a proof system that is well-suited for automation.
Therefore, these methods exist mainly at lower levels of abstraction, often close to an actual implementation.
Such methods are traditionally evaluated empirically on a set of benchmark problems such as the one described by Claessen et. al. \cite{claessen2015}.
Formal explanations backing the observations obtained by the empirical evaluation are still rare.
As of now, it is difficult to classify methods according to their strength and to give theoretical explanations of an empirically observed failure of a given method in a particular context.

% Research program and related work

The work in this article is part of a research program that aims at analyzing methods for \ac{AITP} by applying techniques and results from mathematical logic.
The purpose of this is twofold.
Firstly, formal analyses allow us to complement and to explain the empirical knowledge obtained by the practical evaluations of \ac{AITP} methods.
Secondly, the analyses carried out during this program will inevitably lead to a development of the logical foundations of automated inductive theorem proving.
In particular, we believe that practically relevant negative results are especially valuable in revealing the features a method is lacking.
Thus, negative results may drive the development of new methods.
Moreover, we believe that this research program will strengthen the link between the research in automated inductive theorem proving and mathematical logic, and therefore, may lead to cross-fertilization by providing interesting theoretical techniques from mathematical logic and new problems for mathematical logic.

As part of this research program Hetzl and Wong \cite{wong2017} have given some observations on the logical foundations of inductive theorem proving.
Vierling \cite{vierling2018} has analyzed the n-clause calculus \cite{kersani2013,kersani2014} resulting in an estimate of the strength of this calculus.
Building on this analysis Hetzl and Vierling \cite{hetzl2020} have further abstracted the n-clause calculus and situated this calculus with respect to some logical theories.
The authors are currently also working on an unprovability result for the n-clause calculus.

The subject of \ac{AITP} has recently increasingly focused on integrating mathematical induction in saturation-based theorem provers \cite{kersani2013,kersani2014,cruanes2015,cruanes2017,wand2017,echenim2020,reger2019,hajdu2020}.
In this article we propose abstractions of these systems and investigate how Skolemization interferes with induction in such a system.
In a fairly general yet practically relevant setting we are able to show that Skolem symbols take the role of induction parameters.
We use this insight to provide unprovability results for a family of methods using induction for quantifier-free formulas.
This allows us in particular to obtain unprovability results for the concrete method described in \cite{reger2019,hajdu2020}.

% Structure of the article

In this article we will provide a unified view of a commonly used strategy to integrate induction into saturation-based theorem proving and concentrate on the role of Skolemization in these systems.
To our knowledge the interaction between induction and Skolemization has not been investigated in the related literature.
Section~\ref{sec:preliminaries} introduces all the necessary notations related to our logical formalism, our presentation of Skolemization, and the arithmetic theories used in this article.
We will give a precise presentation of Skolemization, that imposes a concrete naming schema which will be particularly useful in dealing with the languages generated by saturation systems.
In Section~\ref{sec:induction_in_saturation_based_systems} we give an abstract description of saturation-based proof systems and describe abstractly a common strategy to integrate induction in such systems.
We furthermore present a restriction of this system that generalizes a way to handle induction found in most practical saturation systems with induction.
Section~\ref{sec:unrestricted_induction_and_skolemization} gives a very clear characterization of refutation in saturation systems with an unrestricted induction rule (see Theorem~\ref{thm:4}) and analyzes the effects of Skolemization on the induction.
In Section~\ref{sec:restricted_induction_and_skolemization} we analyze the effect of Skolemization in syntactically restricted systems that are closer to the practical methods.
This section culminates in a Skolem-free characterization of these systems (see Theorem~\ref{cor:2}).
Finally in Section~\ref{sec:unprovability} we make use of the results from Section~\ref{sec:restricted_induction_and_skolemization} to provide practically relevant unprovability results for a family of methods using quantifier-free induction formulas (see Theorem~\ref{thm:3}) and apply this result to the concrete method
 presented in~\cite{reger2019,hajdu2020}.

\section{Preliminary Definitions}
\label{sec:preliminaries}
In this section we settle the details of the logical formalism that we use throughout the article.
For the sake of clarity we try to adhere as much as possible to standard terminology, but we introduce some non-standard notations where it is beneficial for the presentation.
In Section~\ref{sec:preliminaries:formulas} we describe our logical formalism and the related notations such as clauses.
Section~\ref{sec:preliminaries:skolemization} introduces some definitions and well-known results related to Skolemization and in particular the naming schema for Skolem symbols that we adopt in this article.
Finally, in Section~\ref{sec:preliminaries:arithmetic} we recall some notions of formal arithmetic and introduce a particular theory of formal arithmetic that will be of use at various occasions.
\subsection{Formulas, theories, and clauses}
\label{sec:preliminaries:formulas}
We work in a setting of classical single-sorted first-order logic with equality.
That is, besides the usual logical symbols we have a logical binary predicate symbol \(=\) denoting equality.
In the context of automated theorem proving it is common to work in a many-sorted setting, but in order to keep the presentation simple we only use one sort.
All our definitions and results easily generalize to the many-sorted case.
A first-order language \(L\) is a countable set of function symbols and predicate symbols with their respective arities.
Let \(\sigma\) be a (function or predicate) symbol, then we write \(\sigma/n\) to denote that \(\sigma\) has arity \(n \in \mathbb{N}\).
Terms are constructed from function symbols and variables.
Formulas are constructed as usual from atomic formulas, the connectives \(\neg\), \(\vee\), \(\wedge\), \(\rightarrow\), and the quantifiers \(\exists\) and \(\forall\).
In order to save some parentheses we assume the following order of precedence for the propositional connectives: \(\neg\), \(\vee\), \(\wedge\), \(\rightarrow\).
By \(\mathcal{F}(L)\) we denote the set of \(L\) formulas.
The notions of bound variables and free variables are defined as usual.
By \(\FV(\varphi)\) we denote the set of free variables of a formula \(\varphi\).
A formula that has no free variables is called a sentence.
By \(\ExistsExOne{y}{\varphi(\vec{x}, y)}\) we abbreviate the formula
\[
  \Exists{y}{\varphi(\vec{x}, y)} \wedge \Forall{y_{1},y_{2}}{(\varphi(\vec{x}, y_{1}) \wedge \varphi(\vec{x}, y_{2}) \rightarrow y_{1} = y_{2})}.
\]

In this article we are more interested in the axioms of a theory, rather than the deductive closure of these axioms.
Hence, we define a theory as a set axioms and manipulate the deductive closure by means of the first-order provability relation (see Definition~\ref{def:25}).
\begin{definition}[Theories]
  \label{def:33}
  Let \(L\) be a first-order language, then a first-order \(L\) theory \(T\) is a set of \(L\) sentences called the axioms of \(T\).
\end{definition}
For the sake of legibility we often present the axioms of a theory as a list of formulas with free variables, with the intended meaning that these formulas are universally closed.
By \(L(T)\) we denote the language of the theory \(T\).
When no confusion arises we sometimes write \(T\) in places where \(L(T)\) is expected.
\begin{definition}[Provability]
  \label{def:25}
  Let \(\varphi\) be a sentence and \(T\) a theory, then we write \(T \vdash \varphi\) to denote that \(\varphi\) is provable in first-order logic from the axioms of \(T\).
  Let \(\Gamma\) be a set of sentences, then we write \(T \vdash \Gamma\) to denote that \(T \vdash \varphi\) for all sentences \(\varphi \in \Gamma\).
  Let \(T_{1}\) and \(T_{2}\) be theories, then we write \(T_{1} \equiv T_{2}\) if \(T_{1} \vdash T_{2}\) and \(T_{2} \vdash T_{1}\).
\end{definition}
Let \(\varphi(\vec{x})\) be a formula and \(T\) a theory, then in order to ease the notation we will sometimes write \(T \vdash \varphi(\vec{x})\) in place of \(T \vdash \Forall{\vec{x}}{\varphi(\vec{x})}\).
\begin{definition}[Conservativity]
  \label{def:26}
  Let \(T_{1}\) and \(T_{2}\) be theories, and \(\Gamma\) a set of formulas.
  We say that \(T_{1}\) is \(\Gamma\)-conservative over \(T_{2}\) (in symbols \(T_{1} \sqsubseteq_{\Gamma} T_{2}\)), if, for all \(\varphi \in \Gamma\), \(T_{1} \vdash \varphi\) implies \(T_{2} \vdash \varphi\).
  We write \(T_{1} \equiv_{\Gamma} T_{2}\) if \(T_{1} \sqsubseteq_{\Gamma} T_{2}\) and \(T_{1} \sqsupseteq_{\Gamma} T_{2}\).
  If \(\Gamma = \mathcal{F}(L)\) for some first-order language \(L\), then we may simply write \(T_{1} \sqsubseteq_{L} T_{2}\) for \(T_{1} \sqsubseteq_{\mathcal{F}(L)} T_{2}\).
\end{definition}
Automated theorem provers---in particular saturation systems---usually do not operate directly on formulas but instead operate on clauses and clause sets (see Section~\ref{sec:induction_in_saturation_based_systems}).
\begin{definition}[Literals and clauses]
  \label{def:29}
  Let \(L\) be a first-order language. An \(L\) literal is an \(L\) atom or the negation thereof.
  An \(L\) clause is a finite set of \(L\) literals.
  An \(L\) clause set is a set of clauses.
  By \(\Box\) we denote the empty clause.
  Let \(C\) and \(D\) be clauses, then we write \(C \vee D\) for the union of the clauses \(C\) and \(D\).
  Let \(\mathcal{C}\) be a clause set and \(D\) a clause, the we write \(\mathcal{C} \vee D\) to denote the clause set \(\{ C \vee D \mid C \in \mathcal{C} \}\).
  Furthermore, we write \(L(\mathcal{C})\) to denote the language of \(\mathcal{C}\), that is, the set of non-logical symbols that occur in clauses of \(\mathcal{C}\).
\end{definition}
Whenever the language \(L\) is clear from the context or irrelevant, we simply speak of clauses and clause sets instead of \(L\) clauses and \(L\) clause sets.

We will now recall basic some model-theoretic concepts and notations.
Let \(L\) be a language, then an \(L\) structure is a pair \(M = (D,I)\), where \(D\) is a non-empty set and \(I\) is an interpretation.
The interpretation \(I\) is a function that assigns to each symbol \(\sigma/k \in L\) an interpretation \(\sigma^{I}\) such that if \(\sigma\) is a predicate symbol, then \(\sigma^{I} \subseteq D^{k}\) and if \(\sigma\) is a function symbol, then \(\sigma^{I} : D^{k} \to D\).
Let \(\varphi(x_{1}, \dots, x_{n})\) be an \(L\) formula and \(d_{1}, \dots, d_{n} \in D\), then we write \(M, \{ x_{i} \mapsto d_{i} \mid  i = 1, \dots, n \} \models \varphi\) if \(\varphi\) is true in \(M\) under the variable assignment that assigns \(d_{i}\) to \(x_{i}\) for \(i = 1, \dots, n\).
\begin{definition}[Notation]
  \label{def:28}
  Let \(L\) be a language, \(M = (D,I)\) an \(L\) structure, then we define \(|M| = D\).
  Moreover, we sometimes write \(d \in M\) if \(d \in D\) and for a symbol \(\sigma \in L\), we also denote its interpretation \(\sigma^{I}\) in \(M\) by \(\sigma^{M}\).
  Let \(\varphi(x_{1}, \dots, x_{n})\) be an \(L\) formula and \(d_{1}, \dots, d_{n} \in D^{|\vec{x}|}\), then we write \(M \models \varphi(d_{1}, \dots, d_{n})\) if \(M, \{ x_{i} \mapsto d_{i} \mid i = 1, \dots, n\} \models \varphi\).
  Furthermore, we write \(M \models \varphi\), if \(M, \{ x_{i} \mapsto d_{i} \mid i = 1, \dots, n \} \models \varphi\), for all \(d_{1}\), \dots, \(d_{n} \in M\).
  Similarly, we write \(M \models C\) for an \(L\) clause \(C\) with free variables \(x_{1}, \dots, x_{n}\), if \(M, \{ x_{i} \mapsto d_{i} \mid i = 1, \dots, n \} \models C\) for all \(d_{1}, \dots, d_{n} \in M\).
  Let \(\Delta\) be a set of formulas and clauses, then we write \(M \models \Delta\) if \(M \models \delta\) for each \(\delta \in \Delta\).
  We write \(\Lambda \models \Delta\) if for every model \(M\) of \(\Lambda\) we have \(M \models \Delta\).
\end{definition}
\begin{definition}
  \label{def:31}
  Let \(L\) be a language and \(M\) a first-order structure, then we define
  \[
    \mathrm{Th}(M) \coloneqq \{ \varphi \mid M \models \varphi, \text{\(\varphi\) is an \(L\) sentence}\}.
  \]
\end{definition}
We are often interested in the formulas that have a certain structure.
\begin{definition}
  \label{def:27}
  We say that a formula is \(\exists_{0}\) (or \(\forall_{0}\) or open) if it is quantifier-free.
  We say that a formula is \(\exists_{n + 1}\) (\(\forall_{n + 1}\)) if it is of the form \(\Exists{\vec{x}}{\varphi(\vec{x}, \vec{y})}\) (\(\Forall{\vec{x}}{\varphi(\vec{x},\vec{y})}\)), where \(\varphi\) is \(\forall_{n}\) (\(\exists_{n}\)) and \(\vec{x}\) is a possibly empty vector of variables.
  Let \(L\) be a first-order language, then by \(\Literal(L)\), \(\Open(L)\), \(\exists_{n}(L)\), and \(\forall_{n}(L)\) we denote the set of literals, open formulas, \(\exists_{n}\) formulas, and \(\forall_{n}\) formulas of the language \(L\).
  We say that a theory is \(\forall_{n}\) (\(\exists_{n})\) if all of its axioms are \(\forall_{n}\) (\(\exists_{n}\)).
\end{definition}
As mentioned above, automated theorem provers often work on sets of clauses, rather than formulas.
Hence, it is necessary to discuss how formulas are associated with clause sets.
In the following definition we fix one such translation that we use throughout the article.
\begin{definition}
  \label{def:conjunctive_normal_form_translation}  
  By \(\cnf\) we denote a fixed function that assigns to any \(\forall_{1}\) sentence \(\varphi\), a clause set \(C_{\varphi}\) such that \(L(\varphi) = L(C_{\varphi})\) and \(\varphi\) and \(C_{\varphi}\) are logically equivalent.
  Let \(T\) be a \(\forall_{1}\) theory, then \(\cnf(T) \coloneqq \bigcup_{\varphi \in T}\cnf(\varphi)\).
\end{definition}
The function \(\cnf\) fixed by the definition above could for example be the translation to conjunctive normal form that proceeds by moving negations inwards and by distributing disjunction over conjunction.
We did not fix this particular translation because it is irrelevant for us how a conjunctive normal form is obtained as long as the translation preserves the language and is logically equivalent to the original sentence.
Since this article focuses on the interaction of induction and Skolemization, we choose to exclude conjunctive normal form translations that do not preserve the language.
The question how these more advanced transformations interact with induction is clearly also important and should be investigated separately.
\subsection{Skolemization}
\label{sec:preliminaries:skolemization}
We essentially use inner Skolemization with canonical names.
On the one hand this form of Skolemization is convenient from a theoretical point of view, because it can be described as a function on formulas.
In particular, the canonical naming schema for Skolem symbols allows us to be precise about the languages generated during the saturation processes considered in this article.
On the other hand, inner Skolemization performs comparatively well with respect to proof complexity \cite{baaz1994}, and furthermore using canonical Skolem symbols does not increase proof complexity.
Hence, this form of Skolemization is also a reasonable choice from the perspective of automated deduction.

We start by defining an operator describing all the Skolem symbols that can be obtained by Skolemizing a single quantifier over a given language \(L\).
This operator is then iterated on the language \(L\) in order to produce all the Skolem symbols that are required to Skolemize \(L\) formulas.
\begin{definition}
  \label{def:3}
  Let \(L\) be a first-order language, then we define
  \begin{equation*}
    \mathfrak{S}_{Q}(L) \coloneqq \{ \mathfrak{s}_{\Quantifier{Q}{x}{\varphi}}/n \mid \text{\(\varphi\) is an \(L\) formula, \(|\FV(\Quantifier{Q}{x}{\varphi})| = n\)} \},
  \end{equation*}
  where \(Q \in \{ \forall, \exists \}\).
  We set \(\mathfrak{S}(L) \coloneqq \mathfrak{S}_{\forall}(L) \cup \mathfrak{S}_{\exists}(L)\).
  Now we define \(\sk(L) \coloneqq L \cup \mathfrak{S}(L)\).
  By \(\sk^{i}(L)\) we denote the \(i\)-fold iteration of the \(\sk\) operation.
  Finally, we define \(\sk^{\omega}(L) \coloneqq \bigcup_{i < \omega}\sk^{i}(L)\).
  We call the stage of a symbol the least \(i \in \mathbb{N}\) such that the symbol belongs to the language \(\sk^{i}(L)\).
  A first-order language \(L\) is Skolem-free if it does not contain any of its Skolem symbols, that is, if \(L \cap \mathfrak{S}(\sk^{\omega}(L)) = \varnothing\).
\end{definition}
Now we can define the universal and existential Skolem form of a formula.
\begin{definition}
  \label{def:4}
  We define the functions \(\skA, \skE: \mathcal{F}(\sk^{\omega}(L)) \to \mathcal{F}(\sk^{\omega}(L))\) mutually inductively as follows
  \begin{align}
    \sk^{Q}(P(\vec{t})) & \coloneqq P(\vec{t}), \nonumber \\
    \sk^{Q}(A \wedge B) & \coloneqq \sk^{Q}(A) \wedge \sk^{Q}(B), \nonumber \\
    \sk^{Q}(A \vee B) & \coloneqq \sk^{Q}(A) \vee \sk^{Q}(B), \nonumber \\
    \sk^{Q}(\neg A) & \coloneqq \neg \sk^{\overline{Q}}(A), \nonumber \\
    \label{eq:37} \sk^{Q}(\Quantifier{Q}{x}{A(x, \vec{y})}) & \coloneqq \sk^{Q}(A(\mathfrak{s}_{\Quantifier{Q}{x}{A(x,\vec{y})}}(\vec{y}),\vec{y})), \tag{*} \\
    \sk^{Q}(\Quantifier{\overline{Q}}{x}{A}) & \coloneqq \Quantifier{\overline{Q}}{x}{\sk^{Q}(A)} \nonumber,
  \end{align}
  for \(Q \in \{\forall, \exists\}\), \(\overline{\forall} = \exists\), \(\overline{\exists} = \forall\), and where in \(\eqref{eq:37}\) \(\vec{y}\) are \emph{exactly} the free variables of \(\Quantifier{Q}{x}{A}\).
  Let \(\Gamma\) be a set of formulas, then we define \(\sk^{Q}(\Gamma) \coloneqq \{ \sk^{Q}(\varphi) \mid \varphi \in \Gamma \}\).
\end{definition}
Before we discuss some details of the \(\skE\) function in more detail, we will look at an example that illustrates how the function \(\skE\) operates.
\begin{example}
   Let \(P/3\) be a predicate symbol, then the existential Skolem form of the sentence \(\Exists{x}{\Forall{y}{{\Exists{z}{P(x, y, z)}}}}\) is given by
  \begin{align*}
    \label{eq:38}
    \skE(\Exists{x}{\Forall{y}{\Exists{z}{P(x, y, z)}}}) & = \skE(\Forall{y}{\Exists{z}{P(c, y, z)}}) \\ & = \Forall{y}{(\skE(\Exists{z}{P(c, y, z)}))}) \\
    & = \Forall{y}{(\skE(P(c, y, f(y))))} = \Forall{y}{P(c, y, f(y))}.
  \end{align*}
  where \(c = \mathfrak{s}_{\Exists{x}{\Forall{y}{\Exists{z}{P(x, y, z)}}}}\) and \(f = \mathfrak{s}_{\Exists{z}{P(c, y, z)}} = \mathfrak{s}_{\Exists{z}{P(\mathfrak{s}_{\Exists{x}{\Forall{y}{\Exists{z}{P(x, y, z)}}}}, y, z)}}\).  
\end{example}
Observe that the symbols that are introduced by \(\skE\) depend on the names of the variables.
Thus, in particular, the symbols introduced for two formulas that only differ in the names of bound variables may not be the same.
For example, let \(P\) be a unary predicate symbol, then \[
  \skE(\Exists{x}{P(x)}) = P(\mathfrak{s}_{\Exists{x}{P(x)}}) \neq P(\mathfrak{s}_{\Exists{y}{P(y)}}) = \skE(\Exists{y}{P(y)}).
\]
Clearly, we could build the equivalence of formulas up to renaming into the Skolemization function.
However, we prefer not to draw logical reasoning into the definition of the Skolemization function.
Identification of provably equivalent formulas can be added by means of additional axioms, such as the Skolem axioms given in Definition~\ref{def:16}.

The following property of Skolemization is well-known.
\begin{proposition}
  \label{pro:3}
  Let \(L\) be first-order language and \(\varphi\) an \(\sk^{\omega}(L)\) formula.
  Then \(\vdash \skE(\varphi) \rightarrow \varphi\) and \(\vdash \varphi \rightarrow \skA(\varphi)\).
\end{proposition}
In general we do not have the converse of the above implications.
We will now introduce Skolem axioms.
These axioms essentially correspond to the existential Skolem form of the logical axioms \(\varphi \rightarrow \varphi\).
\begin{definition}
  \label{def:16}
  Let \(L\) be a first-order language, and \(\varphi(x, \vec{y})\) an \(\sk^{\omega}(L)\) formula, then we define
  \begin{gather*}
    \mathrm{SA}_{x}^{\exists}\varphi \coloneqq \Exists{x}{\varphi \rightarrow \varphi(\mathfrak{s}_{\Exists{x}{\varphi}}(\vec{y}),\vec{y})},
      \\
    \mathrm{SA}_{x}^{\forall}\varphi \coloneqq \varphi(\mathfrak{s}_{\Forall{x}{\varphi}}(\vec{y}), \vec{y}) \rightarrow \Forall{x}{\varphi}.
  \end{gather*}
  We define \(L\text{-}\mathrm{SA} \coloneqq \{ \Forall{\vec{y}}{\SA_{x}^{Q}\varphi} \mid Q \in \{ \forall, \exists \}, \mathfrak{s}_{\Quantifier{Q}{x}{\varphi(x,\vec{y})}} \in \sk^{\omega}(L)\}\).
\end{definition}
The Skolem axioms allow us to also obtain the converse of Proposition~\ref{pro:3}.
\begin{proposition}
  \label{pro:16}
  Let \(L\) be a first-order language, \(\varphi\) an \(\sk^{\omega}(L)\) formula, and \(Q \in \{ \forall, \exists \}\).
  Then we have \(L\text{-}\mathrm{SA} \vdash \varphi \leftrightarrow \sk^{Q}(\varphi)\).
\end{proposition}
\begin{proof}
  Straightforward.
\end{proof}
Skolem axioms over a Skolem-free theory have the following well-known conservation property.
\begin{proposition}
  \label{pro:9}
  Let \(L\) be a Skolem-free first-order language and \(T\) be an \(L\) theory, then \(L\text{-}\mathrm{SA} + T \equiv_{L} T\).
\end{proposition}
With the property above we now immediately obtain the well-known fact that Skolemizing a theory results in a conservative extension of that theory.
\begin{lemma}
  \label{lem:17}
  Let \(L\) be a Skolem-free language and \(T\) be an \(L\) theory, then
  \[
    \skE(T) \equiv_{L} T.
  \]
\end{lemma}
\begin{proof}
  The direction \(\skE(T) \sqsubseteq_{L} T\) is an immediate consequence of Proposition \ref{pro:3}.
  For the other direction we have \(T \equiv_{L}^\text{Prop.~\ref{pro:9}} \SASchema{L} + T \equiv^\text{Prop.~\ref{pro:16}} \SASchema{L} + \skE(T)\).
  Hence \(T \equiv_{L} \skE(T)\).
\end{proof}
This also immediately gives us the following weaker statement that is perhaps more familiar in automated deduction.
\begin{corollary}
  \label{cor:6}
  Let \(L\) be a Skolem-free language and \(T\) be theory, then \(T\) is consistent if and only if \(\skE(T)\) is consistent. 
\end{corollary}

\subsection{Induction and arithmetic}
\label{sec:preliminaries:arithmetic}
We conclude the preliminary definitions with the definition of some notions related to formal arithmetic.
Let us start by discussing the setting for induction that we use in this article.
In automated inductive theorem proving it is customary to work with various inductively defined objects such as the natural numbers, lists, trees, and mutually recursive constructions.
Typically inductive theorem proving concentrates on a multi-sorted setting where a subset of the sorts is interpreted as the term algebra constructed over a set of function symbols, called the constructors.
Such a construction, while of great practical relevance, incurs significant notational complexity. 
Therefore, in order to avoid overloading the presentation, we restrict our setting to the natural numbers.
However, we expect that our results straightforwardly carry over to the more general case mentioned above, because the structure of the induction axiom remains essentially the same.
\begin{definition}
  \label{def:34}
  By \(0/0\) and \(s/1\) we denote the function symbols representing the natural number \(0\)  and the successor function, respectively.
  Moreover, we let \(L_{0} \coloneqq \{ 0/0, s/1 \}\).
\end{definition}
We can now define induction axioms and the first-order structural induction schema.
\begin{definition}
  \label{def:9}
  Let \(L\) be a language, and \(\varphi(x, \vec{z})\) be an \(L\) formula, then the \(L \cup L_{0}\) formula \(\tilde{I}_{x}\varphi\) is given by
  \begin{equation*}
    \label{eq:2}
    \left(\varphi(0, \vec{z}) \wedge \Forall{x}{(\varphi(x, \vec{z}) \rightarrow \varphi(s(x), \vec{z}))}\right) \rightarrow \Forall{x}{\varphi(x, \vec{z})}.
  \end{equation*}
  We refer to the variable \(x\) as the induction variable and to the variables \(\vec{z}\) as the induction parameters.
  Moreover we define the induction axiom \(I_{x}\varphi\) by \(I_{x}\varphi \coloneqq \Forall{\vec{z}}{\tilde{I}_{x}\varphi}\).
  Let \(\Gamma\) be a set of \(L\) formulas, then the set of \(L \cup L_{0}\) sentences \(\IND{\Gamma}\) is given by \(\{ I_{x}\gamma \mid \gamma(x,\vec{z}) \in \Gamma \}\).
\end{definition}
By an arithmetical language we understand a first-order language containing the symbols \(0/0\), \(s/1\), and possibly some symbols representing primitive recursive functions.
In the following definition we recall some standard terminology for arithmetic.
\begin{definition}
  \label{def:7}
  Let \(L\) be an arithmetical language.
  By \(\mathbb{N}_{L}\) the structure whose domain is the set of natural numbers and that interprets the non-logical symbols of \(L\) in the natural way.
  An arithmetical theory is a theory over an arithmetical language.
  Let \(T\) be an \(L\) theory.
  We say that the theory \(T\) is sound if \(\mathbb{N}_{L} \models T\).
  Furthermore, we say that \(T\) is \(\exists_{1}\)-complete if \(\mathbb{N}_{L} \models \varphi\) implies \(T \vdash \varphi\) for all \(\exists_{1}\) \(L\) sentences.
\end{definition}
We conclude this section by describing the setting of linear arithmetic that will in particular serve us in Section~\ref{sec:unprovability:literal_induction_a_case_study} for obtaining unprovability results for the methods \cite{reger2019,hajdu2020}.
The language of linear arithmetic contains besides \(0/0\) and \(s/1\) only the function symbols \(p/1\) and \(+/2\) as infix symbol, where \(p\) denotes the predecessor function and \(+\) denotes the addition.
Clearly, the setting of linear arithmetic is closely related to Presburger arithmetic.
However, we are not interested in the theory of the standard interpretation, but rather in its subtheories such the ones that were already studied by Shoenfield \cite{shoenfield1958}.
This setting of linear arithmetic turns out to be quite useful in the analysis of methods for automated inductive theorem proving, because on the one hand it is simple enough to still allow for straightforward model-theoretic constructions, yet it is complex enough to provide interesting independence results.

Let us fix some notational conventions.
Let \(m \in \mathbb{N}\) and \(t\) be a term, then by \(m \cdot t\) we denote the term \(t + (t + \cdots + (t + t)\cdots)\).
Let \(f\) be a unary function symbol, then \(f^{m}(t)\) stands for \(f(\cdots f(t)\cdots)\).
By \(\numeral{m}\) we denote the term \(s^{m}(0)\).
Our base theory for linear arithmetic is defined as follows.
\begin{definition}
  \label{def:20}
  By \(\mathcal{T}\) we denote the theory axiomatized by the universal closure of the following formulas
  \begin{gather}
    0 \neq s(x), \tag{A1} \label{ax:D} \\
    p(0) = 0, \tag{A2} \label{ax:P:0} \\
    p(s(x)) = x, \tag{A3} \label{ax:P:s} \\
    x + 0 = x, \tag{A4} \label{ax:A:0} \\
    x + s(y) = s(x + y), \tag{A5} \label{ax:A:s}
  \end{gather}
\end{definition}
We conclude with two basic observations about the theory \(\mathcal{T}\).
We shall make use of these observations at several occasions and will for the sake of readability not mention them explicitly every time.
\begin{lemma}
  \label{lem:33}
  \(\mathcal{T} \vdash s(x) = s(y) \rightarrow x = y\).
\end{lemma}
\begin{proof}
  Use \(\eqref{ax:P:s}\).
\end{proof}
\begin{proposition}
  \label{pro:4}
  \(\mathcal{T}\) is sound and \(\exists_{1}\)-complete.
\end{proposition}
\begin{proof}
  The soundness part is obvious.
  For the \(\exists_{1}\)-completeness observe that \(\mathcal{T}\) decides ground formulas.
\end{proof}
\section{Saturation-based systems and induction}
\label{sec:induction_in_saturation_based_systems}
Induction can be integrated into a saturation proving system in different ways.
One possibility is to contain the induction mechanism in a separate module that may use a saturation prover to discharge subgoals.
Moreover, the induction module may receive additional information from the saturation prover, for instance information about failed proof attempts \cite{biundo1986}.
Another, currently more popular, way is to integrate the induction mechanism more tightly into the saturation system as some form of inference rule \cite{kersani2013,kersani2014}, \cite{reger2019,hajdu2020}, \cite{cruanes2015,cruanes2017}, \cite{wand2017}, \cite{echenim2020}.
In this section we give an abstract framework for \ac{AITP} methods integrating induction in saturation-proof systems in terms of a general induction rule.
This framework will allow us to investigate in Sections~\ref{sec:unrestricted_induction_and_skolemization} and~\ref{sec:restricted_induction_and_skolemization} the role of Skolem symbols in these systems.
In Section~\ref{sec:unprovability} we show that the methods described in \cite{reger2019,hajdu2020} fit into our framework.
In Section~\ref{sec:saturation_based_proof_systems} we define saturation systems abstractly and introduce some related notions.
After that, Section~\ref{sec:induction_rules} introduces the notion of induction rule as a general way to integrate induction into a saturation system and presents a practically relevant specialization of this induction rule.

\subsection{Saturation-based proof systems}
\label{sec:saturation_based_proof_systems}
Saturation is a technique of automated theorem proving that consists of computing the closure of a set of formulas or clauses under some inference rules.
The saturation process goes on until some termination condition, such as the derivation of the empty clause, is met or until no more ``new'' formulas can be generated.
Typically saturation-based theorem provers operate in a clausal setting because clauses have less structure and are therefore better suited for automated proof search.

In what follows we concentrate on the refutational setting, because most state-of-the art theorem provers are refutation provers. 
That is, in order to determine for some theory \(T\) whether a given sentence \(\varphi\) is provable in \(T\), the prover saturates the clause set \(\cnf(\skE(T + \neg \varphi))\) until the empty clause is derived.
However our definitions can be easily adapted to the positive case by dualizing them, so as to cover for example connection-like methods.

Practical saturation proof systems are usually based on a variant of the superposition calculus.
In order not to get involved in the technical details of these saturation-based proof systems we will abstractly think of a such a prover as a state transition system whose current state is a set of derived clauses and whose state transitions are inference rules that generate new clauses.
In particular, our notion of saturation system does not have any notion of redundancy mechanisms such as simplification rules and deletion rules.
Since this article is mostly about upper bounds on the logical strength of \ac{AITP} methods, the assumption that clauses are never deleted is unproblematic.
\begin{definition}[Saturation systems]
  \label{def:5}
  A saturation system \(\mathcal{S}\) is a set of inference rules of the form
  \begin{equation*}
    \begin{prooftree}
      \hypo{\mathcal{C}}
      \infer1{\mathcal{D}}
    \end{prooftree},
  \end{equation*}
  also written as \(\mathcal{C}/\mathcal{D}\) where \(\mathcal{C}\) is a set of clauses \(\mathcal{D}\) is a finite set of clauses.
  Let \(\mathcal{S}_{1}\) and \(\mathcal{S}_{2}\) be two saturation-based proof systems, then by \(\mathcal{S}_{1} + \mathcal{S}_{2}\) we denote the system obtained by the union of the inference rules of \(\mathcal{S}_{1}\) and \(\mathcal{S}_{2}\).
\end{definition}
Informally, an inference rule \(\mathcal{C} / \mathcal{D}\) indicates that if the system is in the ``state'' \(\mathcal{C}\), then the system changes into the ``state'' \(\mathcal{C} \cup \mathcal{D}\).
The reason why we consider inference rules of this form is that they allow us to keep track of global properties of the prover such as for example the language of the currently derived clauses.
Observe that our notion of inference rules is very general since \(\mathcal{C}\) may be infinite.
Hence we could formulate an \(\omega\)-rule for saturation systems.
However, we will only work with inference rules that operate with the language of \(\mathcal{C}\) and a finite set of clauses \(\mathcal{C}_{0} \subseteq \mathcal{C}\).
\begin{example}
  The resolution rule can be presented as follows:
  \[
    \begin{prooftree}
      \hypo{\{l \vee C\} \cup \{ m \vee D \} \cup \mathcal{C}}
      \infer1[\(\mathrm{Res}\),]{ \{(C \vee D)\mu \}}
    \end{prooftree}
  \]
  where \(\mathcal{C}\) is a clause set, \(C\) and \(D\) are clauses, and \(\mu\) is the most general unifier of the literals \(l\) and \(\overline{m}\).
\end{example}

\begin{definition}[Deduction, Refutation]
  \label{def:6}
  Let \(\mathcal{C}_{0}\) be a set of clauses and \(\mathcal{S}\) a saturation-based proof system.
  A deduction from \(\mathcal{C}_{0}\) in \(\mathcal{S}\) is a finite sequence of clause sets \(\mathcal{D}_{0}, \dots, \mathcal{D}_{n}\) such that \(\mathcal{C}_{0} = \mathcal{D}_{0}\) and \(\mathcal{D}_{i + 1} = \mathcal{D}_{i} \cup \mathcal{B}_{i}\) such that \(\mathcal{D}_{i}/\mathcal{B}_{i}\) is an inference rule of \(\mathcal{S}\) for \(0 \leq i < n\).
  We say that a clause \(C\) is derivable from \(\mathcal{C}_{0}\) in \(\mathcal{S}\) if there exists a deduction \(\mathcal{D}_{0}, \dots, \mathcal{D}_{n}\) such that \(C \in \mathcal{D}_{n}\).
  A deduction \(\mathcal{D}_{0}, \dots, \mathcal{D}_{n}\) is called a refutation if \(\Box \in \mathcal{D}_{n}\).
\end{definition}
Since we are usually interested in extending saturation systems for pure first-order logic by inference rules for induction we need to introduce the notion of soundness and refutational completeness.
\begin{definition}
  \label{def:30}
  Let \(\mathcal{S}\) be a saturation system.
  We say that \(\mathcal{S}\) is sound if whenever a clause \(C\) is derivable from a clause set \(\mathcal{C}_{0}\) in \(\mathcal{S}\), then \(L(C) \subseteq L(\mathcal{C}_{0})\) and \(\mathcal{C}_{0} \models C\).
  The saturation system \(\mathcal{S}\) is said to be refutationally complete if there is a refutation from \(\mathcal{C}_{0}\) if \(\mathcal{C}_{0}\) is inconsistent.
\end{definition}

\subsection{Induction rules}
\label{sec:induction_rules}
Typically induction is integrated in a saturation prover by a mechanism, that, upon some condition,
 selects some clauses out of the generated clauses and constructs an induction formula based on the selected clauses.
After that, the resulting induction axiom is clausified and the clauses are added to the search space \cite{kersani2013,kersani2014,reger2019,hajdu2020,cruanes2015,wand2017}.
The systems differ in the heuristics that are used to construct the induction formula, in the shape of the resulting induction formulas and in the conditions upon which an induction axiom is added to the search space.
For instance, Kersani and Peltier's method \cite{kersani2013,kersani2014} carries out an induction only once, namely when the generated clauses are sufficient to derive the empty clause.
Thus this method does, technically speaking, not even generate clauses.
We abstract the induction mechanisms of the aforementioned methods by the following induction rule.
\begin{definition}
  \label{def:32}
  The induction rule \(\UINDR\) is given by
  \begin{equation*}
    \begin{prooftree}
      \hypo{\mathcal{C}}
      \infer1[\(\UINDR\)]{\cnf(\sk^{\exists}(I_{x}\varphi(x, \vec{z})))},
    \end{prooftree}
  \end{equation*}
  where \(\mathcal{C}\) is a set of clauses, \(\varphi(x, \vec{z})\) is a \(L(\mathcal{C})\) formula.
\end{definition}
Despite being limited to natural numbers, the induction rule presented above is very general in the sense that it does not impose any restrictions on the complexity of the induction formulas.
None of the methods known to us comes even close to making use of the full power offered by that rule.
Nevertheless, it will serve us as a useful tool for theoretical analyses.

There is an important observation that we can make about this induction rule.
First of all, in a saturation system with this induction rule Skolemization may happen at any time and not just once before the saturation process begins, as is the case in saturation systems for pure first-order logic.
Secondly, the induction rule \(\UINDR\) permits Skolem symbols to appear in induction formulas.
In other words, the induction \(\UINDR\) iteratively extends the language of the induction formulas by Skolem symbols.
Interestingly, a similar situation has been considered in the literature on mathematical logic \cite{beklemishev2003}.
In saturation systems for pure first-order logic, the role of Skolemization is clear: It allows us to obtain an equiconsistent formula without existential quantifiers (see Corollary~\ref{cor:6}).
In saturation systems with the induction rule \(\UINDR\) the role of Skolemization is not clear anymore, in the sense of Corollary~\ref{cor:6}.
This raises the question how the extension of the language of induction formulas by Skolem symbols affects the power of the system.
Also note that this feature is not artificial but actually appears in the concrete methods mentioned above.

We shall address this question in Section~\ref{sec:unrestricted_induction_and_skolemization}.
In particular we will provide a logical characterization of refutability in a sound and complete saturation system extended by the induction rule \(\UINDR\) in terms of a theory with an induction schema (see Theorem~\ref{thm:4}).
As a corollary we obtain the soundness of the rule \(\UINDR\) (see Corollary~\ref{cor:3}).

The following example illustrates how to use the above induction rule.
\begin{example}
  \label{ex:1}
  Let us work in the setting of linear arithmetic and let \(\mathcal{S}\) be a sound and refutationally complete saturation system.
  We will now outline a refutation in \(\mathcal{S} + \UINDR\) of the clause set \(\mathcal{C}_{0}\) given by
  \begin{align*}
    \cnf(\skE(\mathcal{T} + \neg \Forall{x}{\Forall{y}{x + y = y + x}})).
  \end{align*}
  Let \(\skE(\neg \Forall{x}{\Forall{y}{x + y = y + x}}) = (c_{1} + c_{2} \neq c_{2} + c_{1})\), then we have \(c_{1} \in L(\mathcal{C}_{0})\) and
  \begin{equation}
    \label{eq:11}
    \mathcal{C}_{0} \models c_{1} + c_{2} \neq c_{2} + c_{1}.
  \end{equation}
  Let \(\varphi_{1}(x) \coloneqq (c_{1} + x = x + c_{1})\), then we may apply the induction rule \(\UINDR\) to obtain the clause set \(\mathcal{C}_{1} \coloneqq \mathcal{C}_{0} \cup \cnf(\skE(I_{x}\varphi_{1}(x))).\)
  Let \(\skE(I_{x}\varphi_{1}(x)) = \left(\varphi_{1}(0) \wedge (\varphi_{1}(c_{3}) \rightarrow \varphi_{1}(s(c_{3})))\right) \rightarrow \forall x \varphi_{1}(x)\), then we have \(c_{3} \in L(\mathcal{C}_{1})\) and furthermore by \eqref{eq:11} we have
  \begin{equation}
    \label{eq:13}
    \mathcal{C}_{1} \models \neg \varphi_{1}(0) \vee \neg (\varphi_{1}(c_{3}) \rightarrow \varphi_{1}(s(c_{3}))).
  \end{equation}
  Since \(\mathcal{C}_{1} \models c_{1} = c_{1} + 0\), we have \(\mathcal{C}_{1} \models \varphi(c_{1},0) \leftrightarrow c_{1} = 0 + c_{1}\).
  Let \(\varphi_{2}(x) \coloneqq x = 0 + x\), then we apply the induction rule \(\UINDR\) in order to obtain the clause set \(\mathcal{C}_{2} \coloneqq \mathcal{C}_{1} \cup \cnf(\skE(I_{x}\varphi_{2})).\)
  Let \(\skE(I_{x}\varphi_{2}) \coloneqq (\varphi_{2}(0) \wedge (\varphi_{2}(c_{4}) \rightarrow \varphi_{2}(s(c_{4})))) \rightarrow \Forall{x}{\varphi_{2}}\), then by \eqref{eq:13} we have
  \begin{equation}
    \label{eq:14}
    \cnf(\mathcal{C}_{2}) \models \neg \varphi_{2}(0) \vee \neg (\varphi_{2}(c_{4}) \rightarrow \varphi_{2}(s(c_{4}))) \vee \neg (\varphi_{1}(c_{3}) \rightarrow \varphi_{1}(s(c_{3}))).
  \end{equation}
  Now observe that \(\mathcal{T} \models 0 = 0 + 0\) and \(\mathcal{T} \models 0 + s(c_{4}) = s( 0 + c_{4})\).
  Hence, \(\mathcal{T} \models c_{4} = 0 + c_{4} \rightarrow s(c_{4}) = s( 0 + c_{4})\), that is, \(\mathcal{T} \models \varphi_{2}(c_{4}) \rightarrow \varphi_{2}(c_{4})\) and \(\mathcal{T} \models \varphi_{2}(0)\).
  Therefore, by \eqref{eq:14} we obtain
  \begin{equation}
    \label{eq:15}
    \mathcal{C}_{2} \models \neg (\varphi_{1}(c_{3}) \rightarrow \varphi_{1}(s(c_{3}))).
  \end{equation}
  Recall that \(\varphi_{1}(x) = (c_{1} + x = x + c_{1})\).
  Since \(\mathcal{T} \models c_{1} + s(x) = s(c_{3} + x)\), we have by \eqref{eq:15}, \(\mathcal{C}_{2} \models \varphi_{1}(c_{3}) \leftrightarrow s(c_{3} + c_{1}) \neq s(c_{3}) + c_{1}\).
  Let \(\varphi_{3}(x) = (s(c_{3} + x) = s(c_{3}) + x)\), then by the above we obtain
  \begin{equation}
    \label{eq:16}
    \mathcal{C}_{2} \models \neg \varphi_{3}(c_{3}).
  \end{equation}
  Now we apply the induction rule \(\UINDR\) in order to obtain the clause set \(\mathcal{C}_{3} \coloneqq \mathcal{C}_{2} \cup \cnf(\skE(I_{x}\varphi_{3})).\)
  Let \(\skE(I_{x}\varphi_{3}) = (\varphi_{3}(0) \wedge (\varphi_{3}(c_{5}) \rightarrow \varphi_{3}(c_{5}))) \rightarrow \Forall{x}{\varphi_{3}}\), then by \eqref{eq:16} we have
  \begin{equation}
    \label{eq:18}
    \mathcal{C}_{3} \models \neg \varphi_{3}(0) \vee \neg(\varphi_{3}(c_{5}) \rightarrow \varphi_{3}(s(c_{5})).)
  \end{equation}
  Since \(\cnf(\mathcal{T}) \models s(c_{3} + 0) = s(c_{3}) = s(c_{3}) + 0\), we have \(\mathcal{C}_{3} \models \varphi(0)\).
  Moreover, \(\cnf(\mathcal{T}) \models s(c_{3} + s(c_{5})) = s(s(c_{3} + c_{5}))\), hence \(\cnf(\mathcal{T}) \models \varphi_{3}(c_{5}) \rightarrow \varphi_{3}(s(c_{5}))\).
  Hence, by \eqref{eq:18}, we have \(\mathcal{C}_{3} \models \bot\).
  Hence, by the refutational completeness of \(\mathcal{S}\) we obtain a refutation of \(\mathcal{C}_{3}\).
  Therefore, by combining the applications of \(\UINDR\) used to obtain \(\mathcal{C}_{3}\) with the \(\mathcal{S}\) refutation of \(\mathcal{C}_{3}\) we obtain a \(\mathcal{S} + \UINDR\) refutation of \(\mathcal{C}_{0}\).
\end{example}
Analyzing the rule \(\UINDR\) will give us some general insights about the role of Skolem symbols in saturation systems with induction, however in order to be more specific about particular methods we have to consider some restricted forms of this induction rule.
We start by introducing some additional terminology.
We call \emph{initial Skolem symbols} those Skolem symbols that arise from the Skolemization of the input problem and \emph{induction Skolem symbols} those Skolem symbols that are generated by an application of the induction rule.

Before we introduce a restriction of the induction rule that is of practical relevance we will discuss some remarkable design choices encountered in practical methods that we will incorporate into the induction rule:
\begin{itemize}
\item Syntactical restriction of induction formulas:
  The methods presented in \cite{reger2019,hajdu2020} restrict induction formulas to literals, \cite{kersani2013,kersani2014} restricts induction formulas to \(\exists_1\) formulas, and \cite{cruanes2015,cruanes2017} restricts induction formulas to \(\forall_{1}\) formulas.
\item Control over occurrences of Skolem symbols:
  The practical induction mechanisms exert control over occurrences of the induction Skolem symbols either by avoiding the introduction of Skolem symbols altogether \cite{kersani2013,kersani2014} or by introducing nullary Skolem symbols only \cite{reger2019,hajdu2020}, \cite{cruanes2015,cruanes2017}.
  In particular none of these methods allows for parameters in the induction formula.
  As a consequence induction Skolem symbols trivially occur as subterms of ground terms.
\end{itemize}
Restrictions on the shape of the induction formulas is a feature that is common to all methods for automated inductive theorem proving because it is currently still difficult to search efficiently for a syntactically unrestricted induction formula.
We incorporate this feature into the induction rule by parameterizing it by a set of formulas from which the induction formulas are constructed.
The second feature is only slightly more complicated to generalize.
If we are to allow induction formulas with quantifier alternations, then Skolemizing the corresponding induction axioms introduces Skolem symbols that are not nullary.
Hence variables may occur in the scope of induction Skolem symbols.
Therefore we generalize the second feature by explicitly requiring that variables do not occur within the scope of a Skolem symbol.
In other words we require that Skolem symbols may appear in the induction formula only in subterms of ground terms.
Both generalized features are captured by the following restricted induction rule.
\begin{definition}
  \label{def:8}
  Let \(\Gamma\) be a set of formulas, then the rule \(\INDRPF{\Gamma}\) is given by
  \begin{equation*}
    \begin{prooftree}
      \hypo{\mathcal{C}}
      \infer1[\(\INDRPF{\Gamma}\),]{\cnf(\skE(I_{x}\varphi(x, \vec{t})))}
    \end{prooftree}
  \end{equation*}
  where \(\mathcal{C}\) is a set of clauses, \(\varphi(x, \vec{z}) \in \Gamma\), and \(\vec{t}\) is a vector of ground \(L(\mathcal{C})\) terms.
\end{definition}
\begin{remark}
  \label{rem:5}
This restriction on occurrences of Skolem symbols is not only motivated by abstracting the current practice in \ac{AITP}, it is also of independent theoretical interest: As described in~\cite{Dowek08Skolemization}, Skolemization without this restriction in simple type theory makes the axiom of choice derivable, hence this restriction has been introduced in~\cite{Miller87Compact}.
This restriction is also used as an assumption for proving elementary deskolemization of proofs with cut in~\cite{Baaz12Complexity}, \cite{komara2021}.
\end{remark}

Let us again consider an example to illustrate the rule.
\begin{example}
  \label{ex:2}
  Consider the refutation carried out in Example~\ref{ex:1}.
  We have used the induction rule three times to derive the clause sets \(\cnf(\skE(I_{x}c_{1} + x = x + c_{1}))\), \(\cnf(\skE(I_{x}x = 0 + x))\), and \(\cnf(I_{x}s(c_{3} + x) = s(c_{3}) + x)\).
  All three induction formulas are equational atoms in which only nullary Skolem symbols appear.
  Hence the refutation outlined in Example~\ref{ex:1} is also a refutation in \(\mathcal{S} +\INDRPF{\mathrm{Eq}(\mathcal{T})}\), where \(\mathrm{Eq}(L)\) denotes the set of equational atoms over the language \(L\).
\end{example}
As with the rule \(\UINDR\) we now have to ask the question how the system behaves.
There are two major cases that we need to distinguish depending on whether the set of formulas \(\Gamma\) may contain initial Skolem symbols.
By letting \(\Gamma\) be a set of Skolem-free formulas, we can restrict the occurrences of all Skolem symbols in the induction formulas.
In Section~\ref{sec:restricted_induction_and_skolemization} we mainly concentrate on this case and provide a characterization for the refutability in a sound and refutationally complete saturation system with the rule \(\INDRPF{\Gamma}\), thus, settling the question.
In practical systems the initial Skolem symbols usually can appear in the induction formulas without restriction, that is, these systems correspond to the case where the formulas in \(\Gamma\) may contain initial Skolem symbols.
However, this case is actually part of a more general open problem concerning occurrences of Skolem symbols in axiom schemata, that we will not address in the this article (see Remark~\ref{rem:5}).
Nevertheless, we can handle the simple case when the initial Skolem symbols are nullary.
We will mainly deal with this case in Section~\ref{sec:unprovability} in order to provide an unprovability result for the methods described in \cite{reger2019} and \cite{hajdu2020}.
\section{Unrestricted induction and Skolemization}
\label{sec:unrestricted_induction_and_skolemization}
In the previous section we have abstractly described a common integration of induction into a saturation system via the induction rule \(\UINDR\).
In this section we will first represent a sound and refutationally complete saturation system extended by the rule \(\UINDR\) as a logical theory.
After that we make use of this representation in order to investigate the interaction between Skolemization and the induction rule.
\subsection{Representation as logical theory}
\label{sec:unrestricted_induction:characterization}
A useful technique when analyzing \ac{AITP} methods is to reduce the system to an ``equivalent'' logical theory.
Alternatively, when such a theory cannot be found it is a good practice to approximate the system by a logical theory as closely as possible.
The construction of that theory usually reveals the essential features of the method.
Moreover, we can then make use of powerful techniques from mathematical logic in order to study the theory.
In particular, we can compare methods in terms of their representative theories.
\begin{definition}
  \label{def:2}
  Let \(T\) be a theory, then we define the Skolem induction operator \(\SI\) by
  \[
    \SI(T) \coloneqq T + \skE(\IND{L(T)})
  \]
  By \(\SI^{i}(T)\) we denote the \(i\)-fold iteration of \(\SI\) on \(T\).
  Finally, we define \(\SI^{\omega}(T) \coloneqq \bigcup_{i < \omega}\SI^{i}(T)\).
\end{definition}
In the following we will show that the theory \(\SI^{\omega}(T)\) is a faithful representation of a saturation system extended by the induction rule \(\UINDR\) and operating on an initial clause set corresponding to a theory \(T\).
In other words, we will show that for a sound and refutationally complete saturation system \(\mathcal{S}\) and a theory \(T\), the saturation system \(\mathcal{S} + \UINDR\) refutes the clause set \(\cnf(\skE(T))\) if and only if \(\SI^{\omega}(\skE(T))\) is inconsistent.
Intuitively, we can see that this is the case because the operation \(\SI(T)\) corresponds to a simultaneous application of \(\UINDR\) to all \(L(T)\) formulas.
However, by the compactness theorem for first-order logic, only finitely many of these induction formulas actually appear in a proof of the inconsistency of \(\SI^{\omega}(\skE(T))\).
Hence we can derive the same induction axioms with the induction rule \(\UINDR\).
\begin{lemma}
  \label{lem:4}
  Let \(\mathcal{S}\) be a sound saturation system and \(T\) be a theory.
  If \(\mathcal{S} + \UINDR\) refutes \(\cnf(\skE(T))\), then the theory \(\SI^{\omega}(\skE(T))\) is inconsistent.
\end{lemma}
\begin{proof}
  We show the slightly stronger claim that for a \(\mathcal{S} + \UINDR\) deduction \(\mathcal{C}_{0} \subseteq \mathcal{C}_{1} \subseteq \dots \subseteq \mathcal{C}_{j}\) from  \(\cnf(\skE(T))\), we have \(L(\mathcal{C}_{j}) \subseteq L(\SI^{j}(\skE(T)))\) and \(\SI^{\omega}(\skE(T)) \models \mathcal{C}_{j}\).
  We proceed by induction on \(j\).
  For the induction base \(j = 0\) we have \(\SI^{\omega}(\skE(T)) \models \mathcal{C}_{0}\) and \(L(\mathcal{C}_{0}) \subseteq L(\SI^{0}(\skE(T)) = L(\skE(T))\), since \(\mathcal{C}_{0} \subseteq \cnf(\skE(T))\).
  For the induction step we consider the clause set \(\mathcal{C}_{j + 1}\).
  If \(\mathcal{C}_{j + 1}\) is obtained by an inference from \(\mathcal{S}\), then by the soundness of \(\mathcal{S}\) we have \(L(\mathcal{C}_{j + 1}) = L(\mathcal{C}_{j})\) and \(\mathcal{C}_{j} \models \mathcal{C}_{j + 1}\).
  Hence by the induction hypothesis we have \(\SI^{\omega}(\skE(T)) \models \mathcal{C}_{j + 1}\) and clearly \(L(\mathcal{C}_{j + 1}) = L(\mathcal{C}_{j}) \subseteq L(\SI^{j}(\skE(T))) \subseteq L(\SI^{j + 1}(\skE(T)))\).
  If \(\mathcal{C}_{j + 1}\) is obtained by an application of the \(\UINDR\) rule, then \(\mathcal{C}_{j + 1} = \mathcal{C}_{j} \cup \cnf(\skE(I_{x}\varphi(x, \vec{z})))\), where \(\varphi\) is an \(L(C_{j})\) formula.
  Since \(L(C_{j}) \subseteq L(\SI^{j}(\skE(T)))\) we have \(\skE(I_{x}\varphi(x,\vec{z})) \in \SI^{j + 1}(\skE(T))\), hence \(L(C_{j + 1}) \subseteq \SI^{j + 1}(\skE(T))\).
  Moreover since \(\SI^{j}(\skE(T)) \models C_{j}\) we clearly have \(\SI^{j + 1}(\skE(T)) \models C_{j + 1}\).
\end{proof}
\begin{lemma}
  \label{lem:8}
  Let \(\mathcal{S}\) be a refutationally complete saturation-based proof system and \(T\) be a theory.
  If the theory \(\SI^{\omega}(\skE(T))\) is inconsistent, then \(\mathcal{S} + \UINDR\) refutes \(\cnf(\skE(T))\).
\end{lemma}
\begin{proof}
  Assume that \(\SI^{\omega}(\skE(T))\) is inconsistent, then by the compactness theorem there exists a finite subset \(S\) of \(\SI^{\omega}(\skE(T))\) such that \(S\) is inconsistent.
  Furthermore there clearly exist sets \(S_{0}, S_{1}, \dots, S_{n}\) with \(n \in \mathbb{N}\) such that \(S_{0} \subseteq \skE(T)\), \(S \subseteq S_{n}\), and \(S_{i} = S_{i - 1} \cup \{\skE(I_{i})\}\), with \(I_{i} \in \IND{\SI^{i-1}(\skE(T))}\) and \(L(I_{i}) \subseteq L(S_{i})\), for \(i = 1, \dots, n\).

  Now we can easily construct a refutation of \(\cnf(\skE(T))\) in \(\mathcal{S} + \UINDR\) by letting \(\mathcal{C}_{0} = \cnf(\skE(T))\), and obtaining \(\mathcal{C}_{i} = \mathcal{C}_{i - 1} \cup \cnf(\skE(I_{i}))\) for \(i = 1, \dots, n\) by the \(\UINDR\) rule.
  Clearly, \(\mathcal{C}_{n}\) is logically equivalent to \(S_{n}\), therefore we obtain a refutation from \(\mathcal{C}_{n}\) because of the refutational completeness of \(\mathcal{S}\).
\end{proof}
We summarize the results so far in the following proposition.
\begin{proposition}
  \label{pro:6}
  Let \(\mathcal{S}\) be a sound and refutationally complete saturation-based proof system and \(T\) be a theory.
  Then \(\mathcal{S} + \UINDR\) refutes \(\cnf(\skE(T))\) if and only if the theory \(\SI^{\omega}(\skE(T))\) is inconsistent.
\end{proposition}
\begin{proof}
  An immediate consequence of Lemma~\ref{lem:4} and Lemma~\ref{lem:8}.
\end{proof}
The theory \(\SI^{\omega}(\skE(T))\) is still not very convenient to work with.
By working it a bit we can on the one hand eliminate the recursion that interleaves induction and Skolemization and secondly we can even ``factor'' out the Skolemization part.
We start by analyzing which Skolem symbols occur in the theories generated by \(\SI^{\omega}(\cdot)\).
Our first observation is that induction axioms that do not bind a free variable of the inducted upon formula allow us to introduce all the Skolem symbols.
\begin{lemma}
  \label{lem:1}
  \label{lem:9}
  % TODO: Fix the label.
  Let \(\varphi(\vec{y})\) be a formula and \(u\) a variable which does not occur in \(\varphi\).
  Then \(L(\skE(\tilde{I}_{u}\varphi)) = L(\skE( \varphi \rightarrow \varphi))\) and moreover \(\vdash \skE(\tilde{I}_{u}\varphi) \leftrightarrow \skE( \varphi \rightarrow \varphi)\).
\end{lemma}
\begin{proof}
  Since the variable \(u\) does not occur in \(\varphi\), we clearly have
  \begin{align*}
    \skE(\tilde{I}_{u}\varphi)
    & = \skA(\varphi) \wedge \skA(\forall{u}{(\varphi \rightarrow \varphi)}) \rightarrow \skE(\Forall{u}{\varphi}) \\
    & = \skA(\varphi) \wedge (\skE(\varphi) \rightarrow \skA(\varphi)) \rightarrow \Forall{u}{(\skE(\varphi))}.
  \end{align*}
  Since \(\skE(\varphi \rightarrow \varphi) = \skA(\varphi) \rightarrow \skE(\varphi)\) we clearly have \(L(\skE(\tilde{I}_{u}\varphi)) = L(\skE(\varphi \rightarrow \varphi))\).
  Furthermore, \(\skE(\tilde{I}_{u}\varphi)\) clearly is logically equivalent to \(\skE(\varphi \rightarrow \varphi)\).
\end{proof}
The formulas of the form \(\skE(\varphi \rightarrow \varphi)\) are of interest because they correspond, roughly speaking, to Skolem axioms.
\begin{remark}
  \label{rem:3}
  The requirement in Lemma~\ref{lem:9} that the induction formula does not contain the induction variable is peculiar, but convenient to handle.
  A similar result as Lemma~\ref{lem:9} can be achieved without this assumption by working, for example, with induction formulas of the form \({u = u} \; \wedge \; \varphi\), where the variable \(u\) is not free in the formula \(\varphi\).
  In practice a system does usually not intentionally use its induction mechanism to introduce Skolem axioms.
  Instead some systems (for example \cite{cruanes2015,cruanes2017}) provide a lemma rule that introduces the clauses \(\cnf(\skE(\varphi \rightarrow \varphi))\) into the search space.
\end{remark}
\begin{lemma}
  \label{lem:5}
  Let \(T\) be a theory, then \(L(\SI^{\omega}(T)) = \sk^{\omega}(L(T) \cup L_{0})\).
\end{lemma}
\begin{proof}
  The inclusion \(\subseteq\) is obvious.
  For the inclusion \(\supseteq\) observe that
  \[
    L(\SI^{\omega}(T)) = \bigcup_{k < \omega}L(\SI^{k}(T)).
  \]
  Hence, it suffices to show that for every symbol \(\sigma \in \sk^{\omega}(L(T) \cup L_{0})\), there exists \(k \in \mathbb{N}\) such that \(\sigma \in L(\SI^{k + 1}(T))\).
  We proceed by induction on the stage of the symbol \(\sigma\).
  For the base case let \(\sigma\) have stage \(0\), then it belongs to \(L(T) \cup L_{0}\) and we already have \(\sigma \in L(\SI^{1}(T))\).
  Now if \(\sigma \in \sk^{\omega}(L(T) \cup L_{0})\) has stage \(n > 0\), then it is a Skolem symbol of the form \(\sigma = \mathfrak{s}_{\Quantifier{Q}{x}{\varphi}}\) with \(Q \in \{ \forall , \exists \}\) and \(\Quantifier{Q}{x}{\varphi}\) only contains symbols of stage less than \(n\).
  Hence by the induction hypothesis \(L(\Quantifier{Q}{x}{\varphi}) \subseteq L(\SI^{k + 1}(T))\) for some \(k \in \mathbb{N}\).
  Therefore \(\skE(I_{u}\Quantifier{Q}{x}{\varphi}) \in \SI^{k + 2}(T)\), thus by Lemma~\ref{lem:1} the symbol \(\mathfrak{s}_{\Quantifier{Q}{x}{\varphi}}\) belongs to \(L(\SI^{k + 2}(T))\), where \(u\) is a variable that does not occur freely in \(\Quantifier{Q}{x}{\varphi}\).
\end{proof}
With this in mind we see that \(\SI^{\omega}(T)\) contains the existential Skolemization of the \(\sk^{\omega}(L(T))\) induction schema.
This allows us to eliminate the iteration of the operator \(\SI(\cdot)\) that was used to build up the language of the induction.
\begin{lemma}
  \label{lem:6}
  Let \(T\) be a theory, then \(\SI^{\omega}(T) \vdash \skE(\IND{\sk^{\omega}(L(T) \cup L_{0})})\).
\end{lemma}
\begin{proof}
  Let \(\varphi\) be an \(\sk^{\omega}(L(T) \cup L_{0})\) formula.
  By Lemma \ref{lem:5} we have \(L(\SI^{\omega}(T)) = \bigcup_{k < \omega}L(\SI^{k}(T)) = \sk^{\omega}(L(T) \cup L_{0})\).
  Hence, there exists \(k \in \mathbb{N}\) such that \(L(\varphi) \subseteq L(\SI^{k}(T))\).
  Therefore, \(\SI^{k + 1}(T) \vdash \skE(I_{x}\varphi)\).
\end{proof}
Again by using Lemma~\ref{lem:9} it is straightforward to see that by Skolemizing the induction schema \(\IND{\sk^{\omega}(L)}\) we actually obtain all the Skolem axioms.
\begin{lemma}
  \label{lem:10}
  Let \(L\) be a first-order language, then \(\skE(\IND{\sk^{\omega}(L)}) \vdash \SASchema{L}\).
\end{lemma}
\begin{proof}
  Let \(\varphi(x, \vec{y})\) be an \(\sk^{\omega}(L)\) formula and \(u\) be a variable not occurring freely in \(\varphi\).
  Work in \(\skE(\IND{\sk^{\omega}(L)})\), then in particular we have \(\skE(\tilde{I}_{u}(\Forall{x}{\varphi(x,\vec{y})}))\).
  We apply Lemma~\ref{lem:9} to \(\tilde{I}_{u}(\Forall{x}{\varphi})\) in order to obtain
  \[
    \skA(\Forall{x}{\varphi(x, \vec{y})}) \rightarrow \skE(\Forall{x}{\varphi(x, \vec{y})}).
  \]
  By Proposition \ref{pro:3}  we have \(\vdash \skE(\Forall{x}{\varphi}) \rightarrow \Forall{x}{\varphi}\), and hence we obtain
  \[
    \skA(\Forall{x}{\varphi(x, \vec{y})}) \rightarrow \Forall{x}{\varphi(x, \vec{y})}.
  \]
  Now observe that \(\skA(\Forall{x}{\varphi(x, \vec{y})}) = \skA(\varphi(\mathfrak{s}_{\Forall{x}{\varphi(x, \vec{y})}}(\vec{y}), \vec{y}))\). Again by Proposition \ref{pro:3} we have \(\vdash \psi \rightarrow \skA(\psi)\) for all \(\sk^{\omega}(L)\) formulas \(\psi\).
  Therefore, we obtain the desired Skolem axiom
  \[
    \varphi(\mathfrak{s}_{\Forall{x}{\varphi}}(\vec{y}), \vec{y})) \rightarrow \Forall{x}{\varphi}.
  \]
  Now in order to obtain \(\SA_{x}^{\exists}\varphi\) we start with \(\skE(\tilde{I}_{u}(\neg \Exists{x}{\varphi}))\) and apply Lemma~\ref{lem:9} to obtain
  \[
    \skA(\neg \Exists{x}{\varphi}) \rightarrow \skE(\neg \Exists{x}{\varphi}).
  \]
  From this we clearly obtain \(\neg \skE(\Exists{x}{\varphi}) \rightarrow \neg \skA(\Exists{x}{\varphi})\), thus, we get
  \[
    \skA(\Exists{x}{\varphi}) \rightarrow \skE(\Exists{x}{\varphi}).
  \]
  Now by Proposition \ref{pro:3} we first obtain \(\Exists{x}{\varphi} \rightarrow \skE(\varphi(\mathfrak{s}_{\Exists{x}{\varphi}}(\vec{y}), \vec{y}))\).
  Since \(\skE(\Exists{x}{\varphi(x, \vec{y})}) = \skE(\varphi(\mathfrak{s}_{\Exists{x}{\varphi}}(\vec{y}), \vec{y}))\), we get \(\Exists{x}{\varphi} \rightarrow \varphi(\mathfrak{s}_{\Exists{x}{\varphi}}(\vec{y}), \vec{y})\) by another application of \ref{pro:3}.
\end{proof}
\begin{proposition}
  \label{pro:7}
  Let \(T\) be a theory, then
  \[
    \SI^{\omega}(\skE(T)) \equiv \SASchema{(L(T) \cup L_{0})} + T + \IND{\sk^{\omega}(L(T) \cup L_{0})}.
  \]
\end{proposition}
\begin{proof}
  First of all observe that \(\sk^{\omega}(L(\skE(T)) \cup L_{0}) = \sk^{\omega}(L(T) \cup L_{0})\) and therefore \(\SASchema{(L(\skE(T)) \cup L_{0})} = \SASchema{(L(T) \cup L_{0})}\).
  For the direction from right to left we observe that
  \begin{multline*}
      \SASchema{(L(T)\cup L_{0})} + T + \IND{\sk^{\omega}(L(T) \cup L_{0})} \vdash \\ \skE(T) + \skE(\IND{\sk^{\omega}(L(T) \cup L_{0})}).
  \end{multline*}
  With this in mind it is straightforward to see that \(\SASchema{(L(T)\cup L_{0})} + T + \IND{\sk^{\omega}(L(T) \cup L_{0})} \vdash \SI^{\omega}(\skE(T))\).
  For the direction from left to right, we observe that by Lemmas~\ref{lem:6}, \ref{lem:10} we have
  \[
    \SI^{\omega}(\skE(T)) \vdash \SASchema{(L(T) \cup L_{0})} + \skE(T) + \skE(\IND{\sk^{\omega}(L(T) \cup L_{0})}).
  \]
  Hence, by Proposition~\ref{pro:16} we obtain
  \[
    \SI^{\omega}(\skE(T)) \vdash \SASchema{(L(T) \cup L_{0})} + T + \IND{\sk^{\omega}(L(T) \cup L_{0})}. \qedhere
  \]
\end{proof}
As an immediate consequence of the results above we obtain the following characterization of refutability in a sound and refutationally complete saturation based system extended by the induction rule \(\UINDR\).
\begin{theorem}
  \label{thm:4}
  Let \(\mathcal{S}\) be a saturation system, \(T\) a theory, and \(\varphi\) an \(L(T)\) sentence.
  \begin{enumerate}[label=(\roman*)]
  \item If \(\mathcal{S}\) is sound and \(\mathcal{S} + \UINDR\) refutes \(\cnf(\skE(T + \neg \varphi))\), then \[
      \SASchema{(L(T) \cup L_{0})} + T + \IND{\sk^{\omega}(L(T) \cup L_{0})} \vdash \varphi.
    \]
  \item If \(\mathcal{S}\) is refutationally complete and \[
      \SASchema{(L(T) \cup L_{0})} + T + \IND{\sk^{\omega}(L(T) \cup L_{0})} \vdash \varphi,\] then \(\mathcal{S} + \UINDR\) refutes \(\cnf(\skE(T + \neg \varphi))\).
  \end{enumerate}
\end{theorem}
\begin{proof}
  Statement~\textit{(i)} is an immediate consequence of Lemma~\ref{lem:4} and Proposition~\ref{pro:7} and Statement~\textit{(ii)} is an immediate consequence of Lemma~\ref{lem:8} and Proposition~\ref{pro:7}.
\end{proof}
As a corollary we obtain the soundness of the \(\UINDR\) rule with respect to the standard model of an arithmetical language.
\begin{corollary}
  \label{cor:3}
  Let \(\mathcal{S}\) be a sound saturation system, \(L\) an arithmetical language, \(T\) a sound \(L\) theory, and \(\sigma\) an \(L\) sentence.
  If \(\mathcal{S} + \UINDR\) refutes the clause set \(\cnf(\skE(T + \neg \sigma))\), then \(\mathbb{N}_{L} \models \sigma\).
\end{corollary}
\begin{proof}
  By \textit{(i)} of Theorem~\ref{thm:4} it suffices to show that \[\SASchema{L} + T + \IND{\sk^{\omega}(L)} \sqsubseteq_{L} \mathrm{Th}(\mathbb{N}_{L}).\]
  This is shown by expanding \(\mathbb{N}_{L}\) by suitable Skolem functions, just as in the traditional model-theoretic proof of Proposition~\ref{pro:9}.
  The resulting structure satisfies \(\IND{\sk^{\omega}(L)}\) since \(\mathbb{N}_{L}\) has induction for all subsets of \(\mathbb{N}\).
\end{proof}
We conclude this section with a remark.
\begin{remark}
  \label{rem:2}
  In the presence of the Skolem axioms every formula is equivalent to an open formula.
  In particular, for a language \(L\), we have
  \[
    \SASchema{(L\cup L_{0})} + \IND{\Open(\sk^{\omega}(L))} \vdash \IND{\sk^{\omega}(L)}.
  \]
  Thus, we can formulate Theorem~\ref{thm:4} in a slightly more canonical way, by using \(\IND{\Open(\sk^{\omega}(L))}\) in place of \(\IND{\sk^{\omega}(L)}\).
  In other words, in the presence of Skolem axioms Skolem symbols permit us to simulate quantification.
  Conceptually, we can thus split the unrestricted induction rule of Definition~\ref{def:32} into a lemma rule and an induction rule for clause sets. 
\end{remark}
\subsection{Conservativity}
\label{sec:conservativity}
In the previous section we have characterized the extension of a sound and refutationally complete saturation system by the induction rule \(\UINDR\) in terms of a theory with induction over formulas that contain Skolem symbols.
This gives rise to the question how the addition of Skolem symbols to the language of the induction schema affects the strength of the system.
In particular, can we provide an equivalent Skolem-free induction schema?
Let \(L\) be a Skolem-free language and \(T\) an \(L\) theory, then a natural candidate for a Skolem-free characterization of the strength of $\SASchema{L} + T + \IND{\sk^{\omega}(L)}$ is the theory $T + \IND{L}$.
 \begin{question}
   \label{que:3}
  Let \(L\) be a Skolem-free language and \(T\) an \(L\) theory, do we have
  \[
    \SASchema{L} + T + \IND{\sk^{\omega}(L)} \sqsubseteq_{L} T + \IND{L}?
  \]
\end{question}
In the following we give a partial answer to the above question.
The general case remains open.
Our answer relies on the following idea: If a Skolem function is definable in terms of an \(L\) formula then wherever the Skolem symbols occurs we can instead use its definition to eliminate the symbol.
\begin{definition}
  \label{def:11}
  Let \(L\) be a Skolem-free language and \(M\) an \(L\) structure.
  A function \(f : |M|^{k} \to |M|\) is called \(L\)-definable in \(M\) if there exists an \(L\) formula \(\varphi(\vec{x}, y)\) such that for all \(\vec{d} \in |M|^{k}\) we have \(f(\vec{d}) = b\) if and only if \(M \models \varphi(\vec{d}, b)\).
\end{definition}
\begin{definition}
  \label{def:17}
  Let \(L\) be a Skolem-free language.
  We say that an \(L\) structure \(M\) has definable Skolem functions if for every \(L\) formula \(\varphi(\vec{x}, y)\) there exists a function \(f : |M|^{k} \to |M|\) that is \(L\)-definable in \(M\) and
  \[
    \text{\(M \models \Exists{y}{\varphi(\vec{d},y)} \rightarrow \varphi(\vec{d}, f(\vec{d}))\), for all \(\vec{d} \in |M|^{k}\)}.
  \]
\end{definition}
\begin{proposition}
  \label{pro:2}
  Let \(T\) be a Skolem-free theory.
  If every model \(M\) of \(T + \IND{L(T)}\) has definable Skolem functions, then
  \[
    \SASchema{L(T)} + T + \IND{\sk^{\omega}(L(T))} \equiv_{L(T)} T + \IND{L(T)}.
  \]
\end{proposition}
For the sake of the presentation we have moved the proof of Proposition~\ref{pro:2} to Appendix~\ref{sec:appendix:a}.
The proof essentially proceeds by replacing in each model the occurrences of the Skolem symbols by instances of their defining formulas.

In order to illustrate Proposition \ref{pro:2} we will consider some practically relevant special cases.
An important special case of Proposition~\ref{pro:2} is when the Skolem functions are definable already in a theory.
\begin{definition}
  \label{def:18}
  Let \(T\) be a Skolem-free theory.
  We say that \(T\) has definable Skolem functions if for each \(L(T)\) formula \(\varphi(\vec{x}, y)\), there exists an \(L(T)\) formula \(\psi(\vec{x}, y)\) such that
  \[
    T \vdash \Exists{y}{\varphi(\vec{x}, y)} \rightarrow \ExistsExOne{y}{(\psi(\vec{x}, y) \wedge \varphi(\vec{x}, y))}.
  \]
\end{definition}
\begin{proposition}
  \label{pro:13}
  Let \(T\) be a Skolem-free theory with definable Skolem functions, then every model of \(T\) has definable Skolem functions.
\end{proposition}
\begin{proof}
  Let \(\varphi(\vec{x}, y)\) be an \(L(T)\) formula.
  Since \(T\) has definable Skolem function, there exists \(\psi(\vec{x}, y)\) such that \(T \vdash \Exists{y}{\varphi(\vec{x}, y)} \rightarrow \ExistsExOne{y}{(\psi(\vec{x}, y) \wedge \varphi(\vec{x}, y))}\).
  Now let \[
    \psi'(\vec{x}, y) \coloneqq (\neg \Exists{y}{\varphi(\vec{x}, y)} \wedge y = 0) \vee (\Exists{y}{\varphi(\vec{x}, y)} \wedge \psi(\vec{x}, y)).
  \]
  Let us now show that \(T \vdash \ExistsExOne{y}{\psi'(\vec{x}, y)}\).
  We work in \(T\), if \(\Exists{y}{\varphi(\vec{x},y)}\), then there is some \(y\) such that \(\psi(\vec{x}, y)\) and \(\varphi(\vec{x}, y)\).
  Hence we have \(\psi'(\vec{x}, y)\).
  If there is no \(y\) such that \(\varphi(\vec{x}, y)\), then we have \(\psi'(\vec{x}, 0)\).
  Assume that \(\psi'(\vec{x}, y_{1})\) and \(\psi'(\vec{x}, y_{2})\).
  If \(\Exists{y}{\varphi(\vec{x}, y)}\), then we have \(\psi(\vec{x}, y_{1})\) and \(\psi(\vec{x}, y_{2})\), thus, \(y_{1} = y_{2}\).
  Otherwise if \(\neg \Exists{y}{\varphi(\vec{x}, y)}\), then we have \(y_{1} = y_{2} = 0\).
\end{proof}
In particular, a theory has definable Skolem functions if it has a definable well-order.
We simply need to define the Skolem functions in terms of the least of the candidate values in each point.
\begin{definition}
  \label{def:19}
  Let \(L\) be a language, and \(\theta(x,y)\) an \(L\) formula in two variables.
  For the sake of legibility we write \(\theta(x,y)\) as \(x \prec_{\theta} y\) and by \(\ForallBounded{x}{\prec_{\theta}}{y}{\psi(x,y)}\) we abbreviate the formula \(\Forall{x}{( x \prec_{\theta} y \rightarrow \psi(x,y))}\).
  The total order axioms \(\mathrm{TO}_{\theta}\) for \(\theta\) are given by the universal closure of the following formulas
  \begin{gather*}
    \label{eq:12}
    x \not\prec_{\theta} x, \\
    x \prec_{\theta} y \wedge y \prec_{\theta} z \rightarrow x \prec_{\theta} z, \\
    x \prec_{\theta} y \vee y \prec_{\theta} x \vee x = y.
  \end{gather*}
  The least number principle \(L\text{-}\mathrm{LNP}_{\theta}\) for \(\theta(x, y)\) consists of the axioms
  \begin{gather*}
    \label{eq:17}
    \Forall{\vec{z}}{\left(\Exists{x}{\psi(x, \vec{z})} \rightarrow \Exists{x}{(\psi(x, \vec{z}) \wedge \ForallBounded{x'}{\prec_{\theta}}{x}{\neg \psi(x', \vec{z})})}\right)},
  \end{gather*}
  where \(\psi(x,\vec{z})\) is an \(L\) formula.
  We define \(L\text{-}\mathrm{WO}_{\theta} \coloneqq \mathrm{TO}_{\theta} + L\text{-}\mathrm{LNP}_{\theta}\).
\end{definition}
\begin{proposition}
  \label{pro:12}
  Let \(T\) be a Skolem-free theory.
  If there exists an \(L(T)\) formula \(\theta(x,y)\) such that \(T \vdash L(T)\text{-}\mathrm{WO}_{\theta}\), then \(T\) has definable Skolem functions.
\end{proposition}
\begin{proof}
  Let \(\varphi(\vec{x}, y)\) be an \(L(T)\) formula.
  We set \(\psi(\vec{x}, y) = \varphi(x,y) \wedge \ForallBounded{y'}{\prec_{\theta}}{y}{\neg \varphi(\vec{x}, y')}\).
  Now work in \(T\) and assume that \(\Exists{y}{\varphi(\vec{x}, y)}\), then by the least number principle there exists \(y\) such that \(\varphi(\vec{x}, y)\) and moreover \(\ForallBounded{y'}{\prec_{\theta}}{y}{\neg \varphi(\vec{x}, y')}\).
  It remains to show that this \(y'\) is unique.
  Let \(u\) be an element with \(\varphi(\vec{x}, y)\) and \(\ForallBounded{u'}{\prec_{\theta}}{u}{\neg \varphi(u, y)}\).
  If \(u \prec_{\theta} y\), then we obtain \(\neg \varphi(\vec{x}, u)\).
  Analogously we obtain \(\neg \varphi(\vec{x}, y)\) if \(y \prec_{\theta} u\).
  Hence \(u = y\).
\end{proof}
These results are quite far-reaching.
For example, for every sound arithmetic theory \(T\) containing
 the symbol \(+/2\) with the usual primitive recursive definition of \(+\) we have
\[
  T + \IND{L(T)} \vdash L(T)\text{-}\mathrm{WO}_{\theta},
\]
where \(\theta \coloneqq \Exists{z}{x + z = y}\).
Therefore, extending the full induction principle by all the Skolem symbols based on such a theory
 results in a system that proves the same \(L(T)\) formulas as the Skolem-free system.

So far we have considered the effects of extending the full induction schema by all Skolem symbols.
We have concluded that not only is this extension always sound but it is also conservative over the Skolem-free system in a setting where Skolem functions are definable in all models and in particular if the theory provides a well-order.
We have left open the case where there are models in which a Skolem function is not definable.
\section{Restricted induction and Skolemization}
\label{sec:restricted_induction_and_skolemization}
In the previous section we have considered some high-level questions about the soundness and conservativity of Skolemization in saturation theorem proving with an unrestricted induction rule.
In this section we will focus on the role of Skolem symbols in the more practical setting corresponding to the induction rule \(\INDRPF{\Gamma}\) given in Definition~\ref{def:8}, where \(\Gamma\) is a set of formulas.
We start by providing in Section~\ref{sec:restricted_induction_and_skolemization:logical_theory} a representation as a logical theory for sound and refutationally complete saturation systems extended by the induction rule \(\INDRPF{\Gamma}\).
After that we will make use of that characterization in order to clarify the role of the Skolem symbols in saturation systems extended by the rule \(\INDRPF{\Gamma}\) mostly under the assumption that \(\Gamma\) is Skolem-free.
As already mentioned earlier, the restriction to a Skolem-free \(\Gamma\) deviates from practical systems in which \(\Gamma\) may contain initial Skolem symbols but not induction Skolem symbols.
Nevertheless, studying the effect of restricting the occurrences of all Skolem symbols in the induction schema is an interesting theoretical question and allows us to better understand the overall role of Skolem symbols.
\subsection{Representation as logical theory}
\label{sec:restricted_induction_and_skolemization:logical_theory}
We will now provide a preliminary representation as a logical theory for sound and refutationally complete saturation systems extended by the induction rule \(\INDRPF{\Gamma}\).
We start by introducing some additional notions that will be used throughout this section.

So far we have considered the traditional induction schema with induction parameters.
In the following we introduce a notation for induction without induction parameters.
Parameter-free induction schemata have been investigated in mathematical logic \cite{adamowicz1987,kaye1988,beklemishev1997b,cordon2011,jerabek2020}, hence, we adopt a similar notation.
\begin{definition}
  \label{def:parameter_free_induction}
  Let \(\Gamma\) be a set of formulas, then the parameter-free induction schema for \(\Gamma\) formulas \(\INDParameterFree{\Gamma}\) is given by \(\INDParameterFree{\Gamma} \coloneqq \{ I_{x} \gamma \mid \gamma(x) \in \Gamma \}\).
\end{definition}
The grounding operator given in the following definition allows us to obtain all instances of a set of formulas obtained by replacing some of the variables by ground terms.
\begin{definition}
  \label{def:23}
  Let \(\Gamma\) be a set of formulas and let \(L\) be a language.
  Then we define
  \[
    \Gamma \downarrow L \coloneqq \left\{ \gamma(\vec{x}, t_{1}, \dots, t_{n}) \ \mathrel{\bigg|} \
      \begin{split}
         & \gamma(\vec{x},z_{1}, \dots, z_{n}) \in \Gamma,
      \\ & \text{\(t_{1}, \dots, t_{n}\) are ground \(L\) terms}
      \end{split}
      \right\}.
  \]
\end{definition}
We can now introduce an operator corresponding to the rule \(\INDRPF{\Gamma}\).
\begin{definition}
  \label{def:15}
  Let \(T\) be a theory and \(\Gamma\) be a set of formulas.
  \begin{gather*}
    \label{eq:8}
    \GSI{\Gamma}{}{T} \coloneqq T + \skE(\INDParameterFree{(\Gamma \downarrow T)}).
  \end{gather*}
  We define \(\GSI{\Gamma}{i}{T}\) as the \(i\)-fold iteration of the \(\GSI{\Gamma}{}{\cdot}\) operation.
  Finally, we define \(\GSI{\Gamma}{\omega}{T} \coloneqq \bigcup_{i < \omega} \GSI{\Gamma}{i}{T}\).
\end{definition}
It is straightforward to see that \(\GSI{\Gamma}{\omega}{\cdot}\) characterizes a sound and refutationally complete saturation-based proof system extended by the induction rule \(\INDRPF{\Gamma}\).
\begin{proposition}
  \label{pro:5}
  Let \(\mathcal{S}\) be a sound and refutationally complete saturation-based proof system and \(T\) be a theory.
  Then \(\mathcal{S} + \INDRPF{\Gamma}\) refutes \(\cnf(\skE(T))\) if and only if \(\GSI{\Gamma}{\omega}{\skE(T)}\) is inconsistent.
\end{proposition}
\begin{proof}
  Analogous to the proof of Proposition~\ref{pro:6}.
\end{proof}
In Section~\ref{sec:induction_parameters_and_skolem_symbols} we will have a closer look at the role of the Skolem symbols in such saturation systems.

\subsection{Induction parameters and Skolem symbols}
\label{sec:induction_parameters_and_skolem_symbols}
The induction rule \(\INDRPF{\Gamma}\) only generates parameter-free induction axioms, but on the other hand the generated induction axioms may contain Skolem symbols whose role is not yet clear at this point.
Thus, it appears reasonable to begin by comparing sound and refutationally complete saturation systems extended by the rule \(\INDRPF{\Gamma}\) with the induction schema \(\INDParameterFree{\Gamma}\).
In the setting of linear arithmetic with \(\Gamma \coloneqq \Open(\mathcal{T})\) and \(\theta(x,y) \coloneqq  y + x = x \rightarrow y = 0\) we readily obtain an example where both systems differ in strength.
\begin{lemma}
  \label{lem:27}
  Let \(\mathcal{S}\) be a sound and refutationally complete saturation system, then \(\mathcal{S} + \INDRPF{\Open(\mathcal{T})}\) refutes the clause set \(\cnf(\skE(\mathcal{T} + \neg \Forall{x}{\theta(x,x)}))\).
\end{lemma}
\begin{proof}
  By Proposition~\ref{pro:5} it suffices to show the inconsistency of the theory
  \[
    \GSI{\Open}{1}{\skE(\mathcal{T} + \neg \Forall{x}{\theta(x,x)})}.
  \]
  Let \(c \coloneqq \mathfrak{s}_{\Forall{x}{\theta(x,x)}}\), then \(\GSI{\Open}{1}{\skE(\mathcal{T} + \neg \Forall{x}{\theta})} \vdash I_{x}\theta(x,c)\).
  Hence we now work in the theory \(\GSI{\Open}{1}{\skE(\mathcal{T} + \neg \Forall{x}{\theta(x,x)})}\) and proceed by induction on \(x\) in the formula \(\theta(x,c)\).
  For the base case it suffices to see that \(c = c + 0 = 0\) by \(\eqref{ax:A:0}\).
  For the induction step we assume that \(c + x = x \rightarrow c = 0\) and \(c + s(x) = s(x)\).
  By \(\eqref{ax:A:s}\) we obtain \(s( c + x ) = s(x)\) and therefore we obtain \(c + x = x\).
  Hence \(c = 0\) by the assumptions.
  Therefore we now obtain \(\theta(c,c)\) and \(\neg \theta(c,c)\), that is, \(\bot\).
\end{proof}
On the other hand we also have the following.
\begin{lemma}
  \label{lem:13}
  \(\mathcal{T} + \INDParameterFree{\Open(\mathcal{T})} \not \vdash \theta(x,x)\).
\end{lemma}
The proof of Lemma~\ref{lem:13} can be found in Appendix~\ref{sec:appendix:b} and consists of the elimination of the symbol \(p\) from induction formulas followed by the construction of a model \(\mathcal{M}\).
The domain of \(\mathcal{M}\) consists of elements of the form \((b,i) \in \{ 0, 1\} \times \mathbb{Z}\) such that \(b = 0\) implies \(i \in \mathbb{N}\).
Furthermore, the symbol \(0\) is interpreted as the element \((0,0)\) and \(+\) is interpreted as the operation \((b_{1},n_{1}) +^{\mathcal{M}} (b_{2},n_{2}) = (\max\{b_{1},b_{2}\}, n_{1} + n_{2})\).
Hence, \(\mathcal{M} \not \models \theta((1,0),(1,0))\).
\begin{remark}
  \label{rem:1}
  We clearly have \(\mathcal{T} + \IND{\Open(\mathcal{T})} \vdash \theta(x,x)\) by proceeding by induction on \(x\) in the formula \(\theta(x,y)\).
  Hence Lemma~\ref{lem:13} is highly interesting for \ac{AITP} because it provides us with a simple formula that requires induction on a syntactically more complex formula.
\end{remark}
The proof of Lemma~\ref{lem:27} is reminiscent of the obvious proof of \(\theta(x,x)\) in the theory \(\mathcal{T} + \IND{\Open(\mathcal{T})}\).
Thus the proof suggest that the occurrences of Skolem symbols in ground terms of the induction formulas provide some of the strength of induction parameters.
In the following we will confirm this intuition (see Theorem~\ref{thm:1}).

We start by showing that the Skolem symbols appearing in the ground terms of the induction axioms of \(\GSI{\Gamma}{\omega}{\skE(T)}\) are not more powerful than induction parameters.
This is relatively straightforward because ground terms can be abstracted by induction parameters.
In particular, the grounding operation given in Definition~\ref{def:23} is absorbed by parameterized induction.
\begin{lemma}
  \label{lem:41}
  Let \(\Gamma\) be a set of formulas and \(L\) a language, then
  \[
    \IND{\Gamma} \vdash \IND{(\Gamma \downarrow L)}.
  \]
\end{lemma}
\begin{proof}
  Observe that \(\vdash I_{x}\varphi(x,\vec{y},\vec{z}) \rightarrow I_{x}\varphi(x, \vec{y}, \vec{t})\).
\end{proof}
We have announced that this section deals mainly with the case where the set of formulas \(\Gamma\) is Skolem-free.
This corresponds to a saturation system that also restricts the occurrences of the initial Skolem symbols.
In practical systems this is usually not the case, because the restriction mainly applies to induction Skolem symbols.
We briefly address this more general case in the following lemma.
\begin{lemma}
  \label{lem:20}
  Let \(L \supseteq L_{0}\) be a first-order language, \(T\) an \(L\) theory, and \(\Gamma\) a set of \(L\) formulas, then
  \[
    \GSI{\Gamma}{\omega}{T} \sqsubseteq_{L} \SASchema{L} + T + \IND{\Gamma}.
  \]
\end{lemma}
\begin{proof}
  By the compactness theorem it clearly suffices to show that \(\GSI{\Gamma}{n}{T} \sqsubseteq_{L} \SASchema{L} + T + \IND{\Gamma}\) for all \(n \in \mathbb{N}\).
  We proceed by induction on \(n\) and show the slightly stronger claim that \(\SASchema{L} + \GSI{\Gamma}{n}{T} + \IND{\Gamma} \sqsubseteq_{L} \SASchema{L} + T + \IND{\Gamma}\), for all \(n \in \mathbb{N}\).
  The base case is trivial since \(\GSI{\Gamma}{0}{T} = T\).
  For the induction step we have
  \begin{align*}
    \SASchema{L} & + \GSI{\Gamma}{n + 1}{T} \\
    & =_{\text{Def.~\ref{def:15}}} \SASchema{L} + \GSI{\Gamma}{n}{T} + \skE(\IND{\Gamma \downarrow \GSI{\Gamma}{n}{T}}) \\
    & \equiv_{\text{Prop.~\ref{pro:16}}} \SASchema{L} + \GSI{\Gamma}{n}{T} + \IND{\Gamma \downarrow \GSI{\Gamma}{n}{T}} \\
    & \equiv_{\text{Lem.~\ref{lem:41}}} \SASchema{L} + \GSI{\Gamma}{n}{T} + \IND{\Gamma} \\
    & \sqsubseteq_{\text{ind.\ hyp.}} \SASchema{L} + T + \IND{\Gamma}. \qedhere
  \end{align*}
\end{proof}
We can now apply the above lemma to the case that is relevant for us in order to show that allowing occurrences of Skolem symbols in ground terms of induction formulas is not stronger than induction parameters.
\begin{proposition}
  \label{pro:1}
  Let \(L\) be a Skolem-free first-order language, \(T\) an \(L\) theory, and \(\Gamma\) a set of \(L\) formulas, then
  \[
    \GSI{\Gamma}{\omega}{\skE(T)} \sqsubseteq_{L} T + \IND{\Gamma}.
  \]
\end{proposition}
\begin{proof}
  Let \(L' =  L_{0} \cup L \cup L(\skE(T))\), then by Lemma \ref{lem:20} we have \(\GSI{\Gamma}{\omega}{\skE(T)} \sqsubseteq_{L'} \SASchema{L'} + \skE(T) + \IND{\Gamma}\).
  By Proposition \ref{pro:16} we have \(\SASchema{L'} + \skE(T) + \IND{\Gamma} \equiv \SASchema{L'} + T + \IND{\Gamma}\).
  Since \(T + \IND{\Gamma}\) is Skolem-free,  we have by Proposition \ref{pro:9} \(\SASchema{L'} + T + \IND{\Gamma} \sqsubseteq_{L} T + \IND{\Gamma}\).
  Hence \(\GSI{\Gamma}{\omega}{\skE(T)} \sqsubseteq_{L} T + \IND{\Gamma}\).
\end{proof}
In particular this shows that \(\GSI{\Gamma}{\omega}{\skE(T)}\) is not refutationally stronger than the theory \(T + \IND{\Gamma}\).
\begin{corollary}
  \label{cor:1}
  Let \(L\) be a Skolem-free first-order language, \(T\) an \(L\) theory, and \(\Gamma\) a set of \(L\) formulas.
  If \(\GSI{\Gamma}{\omega}{\skE(T)}\) is inconsistent, then \(T + \IND{\Gamma}\) is inconsistent.
\end{corollary}
In the following we will show by a proof-theoretic argument that we even have the converse, that is, ground Skolem terms
 behave in the refutational setting exactly as induction parameters.
 Thus, we start by recalling the necessary concepts from proof theory.
We introduce a partially prenexed form of the induction schema in which the strong quantifier of the induction step is pulled into the quantifier prefix.
Moving this quantifier into the quantifier prefix will simplify the subsequent arguments.
\begin{definition}
  \label{def:14}
  Let \(\gamma(x, \vec{z})\) be a formula, then we define the sentence \(I_{x}'\gamma\) by
  \[
    I_{x}'\gamma \coloneqq \Forall{\vec{z}}{\Exists{x}{\underbrace{\big(\left(\gamma(0, \vec{z}) \wedge (\gamma(x, \vec{z}) \rightarrow \gamma(s(x), \vec{z}))\right) \rightarrow \Forall{w}{\gamma(w, \vec{z})}\big)}_{J_{x}\gamma(x,\vec{z})}}}.
\]
Let \(\Gamma\) be a set of formulas, then we define \(\IND{\Gamma}' \coloneqq \{ I_{x}'\gamma \mid \gamma(x,\vec{z}) \in \Gamma\}\).
\end{definition}
This induction schema is clearly equivalent to the usual one given in Definition~\ref{def:9}.
\begin{lemma}
  \label{lem:16}
   \(\IND{\Gamma} \equiv \IND{\Gamma}'\).
\end{lemma}
We will work with the following Gentzen system, which is essentially a variant of the calculus \(\mathbf{G1c}\) given in \cite{troelstra2000} with atomic logical axioms extended by a cut rule and axioms for equality.
\begin{definition}
  \label{def:22}
  A sequent is an expression of the form \(\Gamma \Rightarrow \Delta\), where \(\Gamma\) and \(\Delta\) are finite multisets of formulas.
\end{definition}
\begin{definition}
  \label{def:35}
  The sequent calculus \(\mathbf{G}\) consists of the following rules

  \noindent Axioms: \\
  
  \begin{minipage}{.49\linewidth}
    \[
    \begin{prooftree}
      \infer0[\(\mathrm{Ax}\)]{A \Rightarrow A}
    \end{prooftree}
    \]
  \end{minipage}
  \begin{minipage}{.49\linewidth}
    \[
      \begin{prooftree}
        \infer0[\(\mathrm{L}\bot\)]{\bot \Rightarrow }
      \end{prooftree}
    \]
  \end{minipage}\smallskip

    \begin{minipage}[t]{.5\textwidth}
    \[
    \begin{prooftree}
      \infer0[\(\mathrm{Refl}\)]{\Rightarrow t = t}
    \end{prooftree}
    \]
  \end{minipage}
  \begin{minipage}[t]{.5\textwidth}
    \[
      \begin{prooftree}
        \infer0[\(\mathrm{Eq}\)]{t = r, A[x/t] \Rightarrow A[x/r]}
      \end{prooftree}
    \]
  \end{minipage} \\

  \noindent Rules for weakening, contraction, and cut: \\

  \begin{minipage}{.49\linewidth}
    \[
      \begin{prooftree}
        \hypo{\Gamma \Rightarrow \Delta}
        \infer1[\(\mathrm{LW}\)]{F, \Gamma \Rightarrow \Delta}
      \end{prooftree}
    \]
    \[
      \begin{prooftree}
        \hypo{F, F, \Gamma \Rightarrow \Delta}
        \infer1[\(\mathrm{LC}\)]{F, \Gamma \Rightarrow \Delta}
      \end{prooftree}
    \]

  \end{minipage}
  \begin{minipage}{.49\linewidth}
    \[
      \begin{prooftree}
        \hypo{\Gamma \Rightarrow \Delta}
        \infer1[\(\mathrm{RW}\)]{\Gamma \Rightarrow \Delta, F}
      \end{prooftree}
    \]
    \[
      \begin{prooftree}
        \hypo{\Gamma \Rightarrow \Delta, F, F}
        \infer1[\(\mathrm{RC}\)]{\Gamma \Rightarrow \Delta, F}
      \end{prooftree}
    \]
  \end{minipage} \smallskip
  \[
    \begin{prooftree}
      \hypo{\Gamma \Rightarrow \Delta, F}
      \hypo{F, \Lambda \Rightarrow \Pi}
      \infer2[\(\mathrm{Cut}\)]{\Gamma, \Lambda \Rightarrow \Delta, \Pi}
    \end{prooftree}
  \]

  \noindent Rules for logical connectives:

  \noindent
  \begin{minipage}[t]{.52\textwidth}
    \[
      \begin{prooftree}
        \hypo{F_{i}, \Gamma \Rightarrow \Delta}
        \infer1[\(\mathrm{L}\wedge_{i}, i = 0,1\)]{F_0 \wedge F_1, \Gamma \Rightarrow \Delta}
      \end{prooftree}
    \]
  \end{minipage}
  \begin{minipage}[t]{.5\textwidth}
    \[
      \begin{prooftree}
        \hypo{\Gamma \Rightarrow \Delta, F}
        \hypo{\Gamma \Rightarrow \Delta, G}
        \infer2[\(\mathrm{R}\wedge\)]{\Gamma \Rightarrow \Delta, F \wedge G}
      \end{prooftree}
    \]
  \end{minipage} \smallskip

  \noindent
  \begin{minipage}{1\linewidth}
    \begin{minipage}[t]{.5\textwidth}
      \[
        \begin{prooftree}
          \hypo{F, \Gamma \Rightarrow \Delta}
          \hypo{G, \Gamma \Rightarrow \Delta}
          \infer2[\(\mathrm{L}\vee\)]{F \vee G, \Gamma \Rightarrow \Delta}
        \end{prooftree}
      \]
    \end{minipage}
    \begin{minipage}[t]{.52\textwidth}
      \[
        \begin{prooftree}
          \hypo{\Gamma \Rightarrow \Delta, F_{i}}
          \infer1[\(\mathrm{R}\vee_{i}\), \(i = 0,1\)]{\Gamma \Rightarrow \Delta, F_0 \vee F_1}
        \end{prooftree}
      \]
    \end{minipage}
  \end{minipage} \smallskip
  
  \begin{minipage}{1\linewidth}
    \vspace{0pt}
    \begin{minipage}[t]{.50\linewidth}
      \vspace{0pt}
      \[
      \begin{prooftree}
        \hypo{\Gamma \Rightarrow \Delta, F}
        \hypo{G, \Gamma \Rightarrow \Delta}
        \infer2[\(\mathrm{L}\rightarrow\)]{F \rightarrow G, \Gamma \Rightarrow \Delta}
      \end{prooftree}
    \]
  \end{minipage}
  \vspace{0pt}
  \begin{minipage}[t]{.50\linewidth}
    \vspace{0pt}
      \[
      \begin{prooftree}
        \hypo{F, \Gamma \Rightarrow \Delta, G}
        \infer1[\(\mathrm{R}\rightarrow\)]{\Gamma \Rightarrow \Delta, F \rightarrow G}
      \end{prooftree}
    \]
    \end{minipage}
  \end{minipage} \smallskip

  \begin{minipage}{1\linewidth}
    \begin{minipage}[t]{.50\linewidth}
      \[
      \begin{prooftree}
        \hypo{F[x/t], \Gamma \Rightarrow \Delta}
        \infer1[\(\mathrm{L}\forall\)]{\Forall{x}{F}, \Gamma \Rightarrow \Delta}
      \end{prooftree}
    \]
    \end{minipage}
    \begin{minipage}[t]{.50\linewidth}
      \[
      \begin{prooftree}
        \hypo{\Gamma \Rightarrow \Delta, F[x/\alpha]}
        \infer1[\(\mathrm{R}\forall\)]{\Gamma \Rightarrow \Delta, \Forall{x}{F}}
      \end{prooftree}
    \]
    \end{minipage}
  \end{minipage} \smallskip

  \begin{minipage}{1\linewidth}
    \begin{minipage}[t]{.50\linewidth}
      \[
      \begin{prooftree}
        \hypo{F[x/\alpha], \Gamma \Rightarrow \Delta}
        \infer1[\(\mathrm{L}\exists\)]{\Exists{x}{F}, \Gamma \Rightarrow \Delta}
      \end{prooftree}
    \]
    \end{minipage}
    \begin{minipage}[t]{.50\linewidth}
      \[
      \begin{prooftree}
        \hypo{\Gamma \Rightarrow \Delta, F[x/t]}
        \infer1[\(\mathrm{R}\exists\)]{\Gamma \Rightarrow \Delta, \Exists{x}{F}}
      \end{prooftree}
    \]
    \end{minipage}
  \end{minipage} \smallskip
  \\ 
  where \(\Gamma, \Delta, \Lambda, \Pi\) stand for multisets of formulas, \(F, G\) stand for formulas, \(A\) stands for atomic formulas, \(t, r\) stand for terms, and for \(R \in \{\mathrm{L}\forall, \mathrm{R}\exists\}\) the variable \(\alpha\) is called the eigenvariable of \(R\) and \(\alpha\) does not occur freely in the conclusion of \(R\).

\end{definition}
We recall some important notions and properties of the calculus \(\mathbf{G}\).
The calculus \(\mathbf{G}\) is sound and complete for first-order logic.
\begin{lemma}
  \label{lem:22}
  Let \(\varphi\) be a sentence, then \(\vdash \varphi\) if and only if there exists a \(\mathbf{G}\) proof of the sequent \(\Rightarrow \varphi\).
\end{lemma}
\begin{proof}
  See for example \cite{troelstra2000}.
\end{proof}
The calculus \(\mathbf{G}\) has the following form of cut elimination.
\begin{definition}
  \label{def:36}
  In a cut inference the formula \(F\) is called the cut formula.
  We say that a \(\mathbf{G}\) proof is in atomic cut-normal form (ACNF, for short) if all of its cut formulas are atomic.
\end{definition}
\begin{lemma}
  \label{lem:23}
  If a sequent \(\Gamma \Rightarrow \Delta\) is provable in \(\mathbf{G}\), then it has a \(\mathbf{G}\) proof in ACNF.
\end{lemma}
\begin{proof}
  See \cite{troelstra2000}.
\end{proof}
\begin{definition}
  \label{def:13}
  The inference rules \(\mathrm{L}\exists\) or \(\mathrm{R}\forall\) are called strong quantifier inference rules.
  Let \(\pi\) be a \(\mathbf{G}\) proof, then by \(\mathrm{sqi}(\pi)\) we denote the number of strong quantifier inferences in \(\pi\).
\end{definition}
In the argument to follow the number of strong quantifier inferences of a proof will be used as the induction measure.
\begin{lemma}
  \label{lem:18}
  Let \(\pi\) be a \(\mathbf{G}\) proof in ACNF of the sequent \(\Pi, \Sigma \Rightarrow \Delta, \Lambda\), then there exists a proof \(\pi'\) in ACNF of \(\Pi, \skE(\Sigma) \Rightarrow \skA(\Delta), \Lambda\) and \(\mathrm{sqi}(\pi') \leq \mathrm{sqi}(\pi)\).
\end{lemma}
\begin{proof}
  We follow the ancestors of the formulas in \(\Sigma\) and \(\Delta\) in \(\pi\) and replace eigenvariables of these ancestors by their respective Skolem terms.
\end{proof}
\begin{proposition}
  \label{pro:14}
  Let \(T\) be a theory with \(L_{0} \subseteq L(T)\) and \(\Gamma\) a set of formulas.
  If \(T + \IND{\Gamma}\) is inconsistent, then \(\GSI{\Gamma}{\omega}{\skE(T)}\) is inconsistent. 
\end{proposition}
\begin{proof}
  Assume that \(T + \IND{\Gamma}\) is inconsistent, then clearly \(\skE(T) + \IND{\Gamma}'\) is inconsistent as well.
  Hence by Lemma~\ref{lem:22} of \(\mathbf{G}\) there exists a proof \(\pi\) in ACNF of a sequent of the form \(\Pi, I \Rightarrow\), where \(\Pi\) is a finite subset of \(\skE(T)\) and \(I\) is a finite subset of \(\IND{\Gamma}'\).
  Observe, furthermore, that we can assume without loss of generality that the symbol \(0\) occurs in \(\Pi\) since \(L_{0} \subseteq L(T)\).
  
  Let \(\mu\) be a proof in ACNF of a sequent of the form \(\Sigma, I \Rightarrow\) with \(\Pi \subseteq \Sigma \subseteq \GSI{\Gamma}{\omega}{\skE(T)}\).
  We proceed by induction on the number of strong quantifier inferences of \(\mu\) in order to obtain a proof of a sequent \(\Sigma' \Rightarrow\) where \(\Sigma' \subseteq \GSI{\Gamma}{\omega}{\skE(T)}\).
  If \(\mu\) does not contain strong quantifier inferences, then we obtain a proof of \(\Sigma \Rightarrow\) by permuting inferences on ancestors of \(I\) downward.
  For the induction step assume that \(\mu\) contains at least one strong quantifier inference.
  Because \(\mu\) does not contain non-atomic cuts, we can permute quantifier inferences toward the bottom of the proof without introducing any new strong quantifier inferences.
  Since \(\Sigma\) is free of strong quantifiers any strong quantifier inference takes place on an ancestor of a formula in \(I\).
  Hence, by permuting a strong quantifier inference toward the bottom of the proof \(\mu\), we obtain a proof \(\nu\) with \(\mathrm{sqi}(\nu) \leq \mathrm{sqi}(\mu)\) of the form
  \begin{equation*}
    \label{eq:4}
    \begin{prooftree}
      \hypo{(\nu'(\alpha))}
      \infer[no rule]1{\Sigma, J_{x}\varphi(\alpha, \vec{t}), I \Rightarrow}
      \infer1[L\(\exists\)]{\Sigma, \Exists{x}{J_{x}\varphi(x, \vec{t})}, I \Rightarrow}
      \infer1[L\(\forall^{*}\)]{\Sigma, \Forall{\vec{z}}{\Exists{x}{J_{x}\varphi(x, \vec{z})}}, I \Rightarrow}
      \infer1[LC]{\Sigma, I \Rightarrow},
    \end{prooftree}
  \end{equation*}
  where \(\varphi(x, \vec{z})\) is a \(\Gamma\) formula and \(\vec{t}\) is a vector of ground terms for which we can assume without loss of generality that \(L(\vec{t}) \subseteq L(\Sigma)\).
  If \(\vec{t}\) would contain a symbol \(\sigma\) of \(I\) that does not already occur in \(\Sigma\), then there is a formula \(\gamma(\vec{x}) \in \Gamma\) containing \(\sigma\) and we introduce \(\skE(I_{x}\gamma(0,\dots,0))\) into \(\Sigma\) by a left weakening.
  Now we let
  \[
    c \coloneqq \mathfrak{s}_{\Forall{x}{\left(\varphi(x, \vec{t}) \rightarrow \varphi(s(x), \vec{t})\right)}}.
  \]
  Then by substituting \(\alpha\) by \(c\) in \(\nu'\) we obtain a proof \(\mu^{\prime} = \nu^{\prime}(c)\) of the sequent \(\Sigma, J_{x}\varphi(c, \vec{t}), I \Rightarrow\).
  We have \(\mathrm{sqi}(\mu^{\prime}) = \mathrm{sqi}(\nu^{\prime}) = \mathrm{sqi}(\nu) - 1\).
  Then by Lemma~\ref{lem:18} there exists a proof \(\mu^{\prime\prime}\) in ACNF of \(\Sigma, \skE(J_{x}\varphi(c, \vec{t})), I \Rightarrow\) with \(\mathrm{sqi}(\mu^{\prime\prime}) \leq \mathrm{sqi}(\mu^{\prime})\).
  Now observe that \(\skE(J_{x}\varphi(c, \vec{t})) = \skE(I_{x}\varphi(x, \vec{t}))\).
  Finally, we apply the induction hypothesis to \(\mu^{\prime\prime\prime}\) in order to obtain the desired proof.

  Now we finish by applying the above procedure to \(\pi\) in order to obtain a proof of \(\Sigma \Rightarrow\) with \(\Pi \subseteq \Sigma \subseteq \GSI{\Gamma}{\omega}{\skE(T)}\).
  By Lemma~\ref{lem:22} it follows that \(\GSI{\Gamma}{\omega}{\skE(T)}\) is inconsistent.
\end{proof}
We can summarize the results in the following proposition.
\begin{proposition}
  \label{thm:1}
  Let \(L\) be a Skolem-free first-order language, \(T\) an \(L\) theory with \(L_{0} \subseteq L(T)\), and \(\Gamma\) a set of \(L\) formulas, then \(\GSI{\Gamma}{\omega}{\skE(T)}\) is inconsistent if and only if \(T + \IND{\Gamma}\) is inconsistent.
\end{proposition}
\begin{proof}
  An immediate consequence of the propositions~\ref{pro:1} and~\ref{pro:14}.
\end{proof}
The above result shows that in a refutational setting allowing Skolem symbols to appear in ground terms of induction formulas corresponds exactly to induction with parameters.
This confirms our initial intuition that Skolem symbols in ground terms behave like induction parameters.
We can rephrase the result of Proposition~\ref{thm:1} as follows.
\begin{theorem}
  \label{cor:2}
  Let \(L\) be a Skolem-free first-order language, \(T\) an \(L\) theory, \(\Gamma\) a set of \(L\) formulas, \(\varphi\) an \(L\) formula such that \(L_{0} \subseteq L(T) \cup L(\varphi)\), and \(\mathcal{S}\) a sound and refutationally complete saturation system.
  Then \(\mathcal{S} + \INDRPF{\Gamma}\) refutes \(\cnf(\skE(T + \neg \varphi))\) if and only if \(T + \IND{\Gamma} \vdash \varphi\).
\end{theorem}
We have thus obtained a Skolem-free characterization of a sound and refutationally complete saturation-based proof system with the induction rule \(\INDRPF{\Gamma}\).
We conclude this section with a question about a generalization of Theorem \ref{cor:2}.
\begin{question}
  \label{que:1}
  Consider again the situation of Lemma~\ref{lem:20}, where we have shown that \(\GSI{\Gamma}{\omega}{T}\) is \(L\) conservative over \(\SASchema{L} + T + \IND{\Gamma}\) where \(L \supseteq L_{0}\) is a first-order language, \(T\) an \(L\) theory, and \(\Gamma\) a set of \(L\) formulas.
  This gives rise to the question whether we can characterize a system that allows initial Skolem symbols to occur in the induction formulas without restriction, but restricts the occurrences of induction Skolem symbols in an analogous way to Proposition~\ref{pro:14}.
  In particular, is \(\GSI{\Gamma}{\omega}{T}\) inconsistent if and only if \(\SASchema{L} + T + \IND{\Gamma}\) is inconsistent?
\end{question}
\section{Unprovability}
\label{sec:unprovability}
In the previous sections we have studied two forms of induction rules occurring in saturation-based induction provers.
In particular we were able to give a Skolem-free characterization as a logical theory of the induction rule \(\INDRPF{\Gamma}\) where \(\Gamma\) is a set of Skolem-free formulas.
In this section we will make use of this result in order to provide concrete unprovability results for saturation systems that make use of this induction rule.
In Section~\ref{sec:unprovability:open_induction} we will provide unprovability results for saturation-based systems that are based on the induction rule \(\INDRPF{\Open(L)}\), where \(L\) stands for the language of the initial clause set.
Then in Section~\ref{sec:unprovability:literal_induction_a_case_study} we show that the concrete methods described in \cite{reger2019,hajdu2020} belong to this family and that therefore we obtain unprovability results for these methods.
\subsection{Open induction}
\label{sec:unprovability:open_induction}
The setting of linear arithmetic described in Section~\ref{sec:preliminaries:arithmetic} proves to be a source of very simple and practically relevant unprovability examples.
We make use of an elegant characterization proved by Shoenfield \cite{shoenfield1958}.
\begin{definition}
  \label{def:21}
  By \(\mathcal{T}'\) we denote the theory having the axioms of \(\mathcal{T}\) together with the axioms
  \begin{gather}
    x \neq 0 \rightarrow x = s(p(x)) \tag{B1} \label{B:1}, \\
    x + y = y + x \tag{B2} \label{B:2}, \\
    (x + y) + z = x + (y + z) \tag{B3} \label{B:3}, \\
    x + y = x + z \rightarrow y = z. \tag{B4} \label{B:4}
  \end{gather}
\end{definition}
\begin{theorem}[Shoenfield \cite{shoenfield1958}]
  \label{thm:9}
  \(\mathcal{T}' \equiv \mathcal{T} + \IND{\Open(\mathcal{T})}\).
\end{theorem}
The following formulas were already studied by Shoenfield in \cite{shoenfield1958}.
Their interesting relation to the theory \(\mathcal{T}'\) will be crucial for our unprovability results.
\begin{definition}
  Let \(m\) and \(n\) be natural numbers, then we define
  \label{def:10}
  \begin{gather*}
    \label{eq:6}
    \text{\(C_{m} \coloneqq \Forall{x, y}{(m \cdot x = m \cdot y \rightarrow x = y)}\)}\\
    \text{\(D_{m,n} \coloneqq \Forall{x, y}{(s^{n}(m \cdot x) \neq m \cdot y)}\).}
  \end{gather*}
\end{definition}
The proof of \cite[Theorem 2]{shoenfield1958} given by Shoenfield can be seen to show that \(\mathcal{T} + \IND{\Open(\mathcal{T})}\) does not prove any of the formulas \(C_{m}\) and \(D_{m,n}\) with \(0 < n < m\).
\begin{lemma}[{\cite{shoenfield1958}}]
  \label{lem:11} \label{lem:12}
  Let \(0 < n < m\), then
  \begin{itemize}
  \item \(\mathcal{T} + \IND{\Open(\mathcal{T})} \not \vdash C_{m}\);
  \item \(\mathcal{T} + \IND{\Open(\mathcal{T})} \not \vdash D_{m,n}\).
  \end{itemize}
\end{lemma}
We have now everything at hand to formulate the unprovability result.
\begin{definition}
  \label{def:1}
  Let \(m, n \in \mathbb{N}\), then the clause sets \(\mathcal{X}_{m}\) and \(\mathcal{Y}_{m,n}\) are given by
  \begin{gather*}
    \mathcal{X}_{m} \coloneqq \cnf(\skE(\mathcal{T}' + \neg C_{m})), \\
    \mathcal{Y}_{m,n} \coloneqq \cnf(\skE(\mathcal{T}' + \neg D_{m,n})).
  \end{gather*}
\end{definition}
\begin{theorem}
  \label{thm:3}
  Let \(\mathcal{S}\) be a sound saturation system and \(\mathcal{C} \in \{ \mathcal{X}_{m}, \mathcal{Y}_{m,n} \mid 0 < n < m\}\), then \(\mathcal{S} + \INDRPF{\Open(L(\mathcal{C}))}\) does not refute the clause set \(\mathcal{C}\).
\end{theorem}
\begin{proof}
  We consider the case for \(\mathcal{C} = \mathcal{X}_{m}\) with \(1 < m\).
  The other case is treated analogously.
  Proceed indirectly and assume that \(\mathcal{S} + \INDRPF{\Open(L(\mathcal{X}_{m}))}\) refutes \(\mathcal{X}_{m}\).
  Then by Lemma~\ref{lem:20} we have
  \[
    \SASchema{L(\mathcal{T})} + \skE(\mathcal{T}') + \IND{\Open(L(\mathcal{X}_{m}))} \vdash \skA(C_{m})
  \]
  First of all, observe that \(\skE(\mathcal{T}') = \mathcal{T}'\).
  By applying Proposition~\ref{pro:16} we obtain
  \[
    \SASchema{L(\mathcal{T})} + \skE(\mathcal{T}') + \IND{\Open(L(\mathcal{X}_{m}))} \vdash C_{m}.
  \]
  Now observe that since \(L(\mathcal{X}_{m})\) extends \(L(\mathcal{T})\) only by constants, we have \(\Open(L(\mathcal{X}_{m})) = \Open(L(\mathcal{T}))\downarrow L(\mathcal{X}_{m})\) and therefore by Lemma \ref{lem:41} we obtain
  \[
    \SASchema{L(\mathcal{T})} +  \mathcal{T}' + \IND{\Open(L(\mathcal{T}))} \vdash C_{m}.
  \]
  Hence, \(\mathcal{T}'\), \(\IND{\Open(L(\mathcal{T}))}\) and \(C_{m}\) are Skolem-free, thus, we can apply Proposition~\ref{pro:9} to obtain
  \[
    \mathcal{T}' + \IND{\Open(L(\mathcal{T}))} \vdash C_{m}.
  \]
  By Theorem \ref{thm:9} we furthermore obtain \(\mathcal{T} + \IND{\Open(L(\mathcal{T}))} \vdash C_{m}\).
  This contradicts Lemma \ref{lem:11}.
\end{proof}
This result begs the question which features a system needs in order to prove the sentences \(C_{m}\) and \(D_{m,n}\) for \(0 < n < m\).
In the following we briefly mention some extensions of the open induction schema that would allow us to overcome our unprovability results.
The extensions we suggest are purely theoretical in the sense that we do not take into account whether they can be implemented efficiently in a saturation system.
A possible extension follows from a remark by Shoenfield \cite{shoenfield1958} that \(C_{m}\) and \(D_{m,n}\) with \(0 < n < m\) can be proved with parameterized double induction (also known as simultaneous induction) on open formulas.
\begin{definition}
  \label{def:24}
  Let \(\gamma(x,y,\vec{z})\) be a formula, then the formula \(\tilde{I}_{(x,y)}\gamma\) is given by
  \begin{multline*}
    \label{eq:10}
        \left(\Forall{x}{ \gamma(x, 0, \vec{z})}
        \wedge \Forall{y}{\gamma(0, y, \vec{z})}
        \wedge \Forall{x,y}{(\gamma(x,y, \vec{z}) \rightarrow \gamma(s(x), s(y), \vec{z}))}\right)
        \\
        \rightarrow \Forall{x,y}{\gamma(x, y, \vec{z})}.
  \end{multline*}
    Let \(\Gamma\) be a set of formulas, then the double induction schema \(\IND{\Gamma}_{2}\) for \(\Gamma\) formulas is given by \(\IND{\Gamma}_{2} \coloneqq \{ \Forall{\vec{z}}{\tilde{I}_{(x,y)}}\gamma \mid \gamma(x,y,\vec{z}) \in \Gamma\}\).
\end{definition}
\begin{lemma}
  \label{lem:3}
  Let \(m,n \in \mathbb{N}\) with \(0 < n < m\), then
  \begin{enumerate}[label={(\roman*)},itemindent=1em]
    \item \(\mathcal{T} + \DIND{\Open(\mathcal{T})} \vdash C_{m}\);
    \item \(\mathcal{T} + \DIND{\Open(\mathcal{T})} \vdash D_{m, n}\).
  \end{enumerate}
\end{lemma}
\begin{proof}
  Straightforward.
\end{proof}
The second possibility is to extend the induction rule used by the system at least to \(\forall_{1}\) formulas without parameters.
\begin{lemma}
  \label{lem:2}
  Let \(m,n \in \mathbb{N}\) with \(0 < n < m\), then
  \begin{enumerate}[label={(\roman*)},itemindent=1em]
  \item \(\mathcal{T} + \INDParameterFree{\forall_{1}(\mathcal{T})} \vdash \mathcal{T}'\); \label{lem:2:1}
  \item \(\mathcal{T} + \INDParameterFree{\forall_{1}(\mathcal{T})} \vdash C_{m}\); \label{lem:2:2}
  \item \(\mathcal{T} + \INDParameterFree{\forall_{1}(\mathcal{T})} \vdash D_{m,n}\). \label{lem:2:3}
  \end{enumerate}
\end{lemma}
\begin{proof}
  The proof of~\ref{lem:2:1} is left as an exercise.
  For~\ref{lem:2:2} we work in \(\mathcal{T} + \INDParameterFree{\forall_{1}(\mathcal{T})}\) and proceed by induction on the formula \(\Forall{y}{(m \cdot x = m \cdot y \rightarrow x = y)}\).
  For the base case we have to show that \(m \cdot 0 = m \cdot y \rightarrow 0 = y\).
  By Lemma~\ref{pro:4} we have \(m \cdot 0 = 0\).
  By \(\eqref{B:1}\) we need to distinguish two cases.
  If \(y = 0\), then we are done, otherwise we obtain a contradiction by \(\eqref{ax:D}\).
  For the induction step we assume \(\Forall{y}{\left(m \cdot x = m \cdot y \rightarrow x = y\right)}\) and \(m \cdot s(x) = m \cdot y\).
  We want to obtain \(s(x) = y\).
  By \(\eqref{ax:A:s}\) and \(\eqref{B:2}\) we obtain \(s^{m}(m \cdot x) = m \cdot s(x) = m \cdot y\).
  By \(\eqref{B:1}\) we can distinguish two cases.
  If \(y = 0\), then by~\ref{pro:4} we \(s^{m}(m \cdot x) = 0\), which contradicts \(\eqref{ax:D}\).
  Hence by Lemma~\ref{lem:33} we have \(m \cdot x = m \cdot p(y)\) and it suffices to show \(x = p(y)\).
  By the induction hypothesis we have \(m \cdot x = m \cdot p(y) \rightarrow x = p(y)\).
  Thus we obtain \(x = p(y)\).

  For~\ref{lem:2:3} we proceed analogously.
\end{proof}
Shoenfield has shown the following interesting theorem.
\begin{theorem}[\cite{shoenfield1958}]
  \label{thm:5}
  \(\mathcal{T}' + \{ C_{m} \mid 1 < m \} + \{ D_{m,n} \mid 0 < n < m\}\) is complete for quantifier-free formulas.
\end{theorem}
From this it follows that at least in the setting of linear arithmetic double induction and parameter-free \(\forall_{1}\) induction are sufficient to prove all true quantifier-free formulas.

In a similar way to what we did in this section we obtain many more unprovability results by using independence results of Shepherdson \cite{shepherdson1964} and Schmerl \cite{schmerl1988}.
However, these results are formulated in the language that besides the symbols of linear arithmetic contains the symbols \(\dot{-}/2\) and \(\cdot/2\) for the truncated subtraction and multiplication, respectively.
The properties that are shown independent of the base theory with open induction express slightly more complicated properties such as the irrationality of the square root of two, Fermat's last theorem for \(n = 3\), and similar diophantine equations.
Hence, these independence results are currently less practically realistic.

\subsection{Literal induction: a case study}
\label{sec:unprovability:literal_induction_a_case_study}
In the previous section we have provided unprovability results for sound saturation systems that are extended by the rule \(\INDRPF{\Open(L)}\), where \(L\) is a Skolem-free language.
In this section we will show that these results apply to the concrete systems described in \cite{reger2019,hajdu2020}.

In \cite{reger2019} Reger and Voronkov describe an \ac{AITP} system that extends a sound saturation-based proof system by the induction rule
\begin{equation*}
  \begin{prooftree}
    \hypo{\{\overline{l}(a) \vee C\} \cup \mathcal{C}}
    \infer1[\(\LiteralAINDR_{1}\)]{\cnf(\neg\skA(l(0) \wedge \Forall{x}{(l(x) \rightarrow l(s(x)))})) \vee C}
  \end{prooftree}
\end{equation*}
where \(a\) is a constant, \(l(x)\) is a \emph{literal free of \(a\)}, and \(\overline{l}(a)\) ground.
We informally refer to this induction rule as the first analytical literal induction rule.
Basically, this induction rule operates as follows:
Whenever a clause of the form \(\overline{l}(a) \vee C\) is encountered, then the rule generates the clauses corresponding to the induction axiom \(I_{x}l(x)\) and immediately resolves these against \(\overline{l}(a) \vee C\).
In a practical implementation the rule will not apply to every clause of the form \(\overline{l}(a) \vee C\) but only when some additional conditions are satisfied.
We call this induction rule analytical because an induction is carried out only for literals that actually are generated during the saturation process.
The motivation for choosing the very restricted induction rule \(\LiteralAINDR_{1}\) is to solve problems that require ``little'' induction reasoning and complex first-order reasoning \cite{reger2019}.
In particular the induction rule is chosen so as to not generate too many clauses, which otherwise would  potentially result in performance issues.
Empirical observations \cite{hajdu2020}, however, suggest that this method is unable to deal even with very simple yet practically relevant problems such as
\[
  x + (x + x) = (x + x) + x.
\]
In order to relax the overly restricting analyticity, \cite{hajdu2020} introduces the following induction rule:
\begin{equation*}
  \begin{prooftree}
    \hypo{\{\overline{l}(a) \vee C\} \cup \mathcal{C}}
    \infer1[\(\LiteralAINDR_{2}\)]{\cnf(\neg \skA\left(l(0) \wedge \Forall{x}{(l(x) \rightarrow l(s(x)))}\right)) \vee C}
  \end{prooftree}
\end{equation*}
where \(\overline{l}(x)\) is a literal, \(a\) is a constant such that \(l(a)\) is ground.
This rule reduces the degree of analyticity by allowing induction to be carried out on slight generalizations of the currently derived literals.
This results in more possibilities to add induction axioms to the search space and thus makes search more difficult, but the degree of analyticity of the induction is reduced sufficiently to make the method able to prove some challenging formulas such as for example \(x + (x + x) = (x + x) + x\) (See \cite{hajdu2020} for details).
It is clear that the rule \(\LiteralAINDR_{2}\) is at least as strong as the rule \(\LiteralAINDR_{1}\).
Hence we will in the following concentrate on the rule \(\LiteralAINDR_{2}\).

In the next step we will show how the induction rule \(\LiteralAINDR_{2}\) can be expressed in terms of the restricted induction rule given in Definition \ref{def:8}.
The proof proceeds in three steps: First we extract the induction axioms that are introduced with \(\LiteralAINDR_{2}\); secondly, we derive these induction axiom with the induction rule of Definition \ref{def:8}; finally, we use first-order inferences to reconstruct a refutation.
\begin{lemma}
  \label{lem:7}
  Let \(\mathcal{S}\) be a sound saturation system and \(\mathcal{D}_{0}, \dots, \mathcal{D}_{n}\) and \(\mathcal{S} + \LiteralAINDR_{2}\) deduction.
  There exist \(L(\mathcal{D}_{0})\) literals
  \[
    \left(l_{i}(x, c_{1}, \dots, c_{i - 1})\right)_{i = 1, \dots, k},
  \]
  where \(n < k\) and \(c_{j} = \mathfrak{s}_{\Forall{x}{(l_{j}(x, c_{1}, \dots, c_{j - 1}) \rightarrow l_{j}(s(x), c_{1}, \dots, c_{j - 1}))}}\) for \(0 < j \leq k \), such that \(L(\mathcal{D}_{n}) \subseteq L(\mathcal{D}_{0}) \cup \{ c_{1}, \dots, c_{k}\} \) and
  \[
    \mathcal{D}_{0} \cup \skE(\{I_{x}l_{i}(x,c_{1}, \dots, c_{i-1}) \mid 0 < i \leq k \}) \models \mathcal{D}_{n}.
  \]
\end{lemma}
\begin{proof}
  We start with a straightforward observation.
  Since the induction rule \(\LiteralAINDR_{2}\) introduces only nullary Skolem symbols, every literal appearing in an \(\mathcal{S} + \LiteralAINDR_{2}\) deduction from \(\mathcal{D}_{0}\) is of the form \(l(\vec{x},\vec{c})\), where \(\vec{c}\) is a vector of induction Skolem symbols and \(L(l(\vec{x},\vec{y})) \subseteq L(\mathcal{D}_{0})\).

  Now let us proceed by induction on the length \(n\) of the deduction from \(\mathcal{D}_{0}\).
  If \(n = 0\), then clearly we are done.
  Now assume that the claim holds up to \(\mathcal{D}_{n}\) and consider \(\mathcal{D}_{n + 1}\).
  If \(\mathcal{D}_{n + 1}\) is derived by from \(\mathcal{D}_{n}\) by an inference from \(\mathcal{S}\), then by the soundness of \(\mathcal{S}\) we have \(L(\mathcal{D}_{n + 1}) \subseteq L(\mathcal{D}_{n})\) and \(\mathcal{D}_{n} \models \mathcal{D}_{n + 1}\).
  Hence, we are done by applying the induction hypothesis to \(\mathcal{D}_{n}\).
  If \(\mathcal{D}_{n + 1}\) is derived from \(\mathcal{D}_{n}\) by \(\LiteralAINDR_{2}\), then \(\mathcal{D}_{n}\) contains a clause \(\overline{l}(a,c_{1}, \dots, c_{k}) \vee C\) and \(\mathcal{D}_{n + 1} = \mathcal{D}_{n} \cup \cnf(\skE(I_{x}l(x, c_{1}, \dots, c_{k}))\) and moreover \(L(\mathcal{D}_{n}) \subseteq L(\mathcal{D}_{0}) \cup \{ 1, \dots, n \}\).
  We let \(l_{k + 1}(x, \vec{y}) \coloneqq l(x, \vec{y})\) and \(c_{k + 1} = \mathfrak{s}_{\Forall{x}{(l_{k}(x, c_{1}, \dots, c_{k}) \rightarrow l_{k}(s(x), c_{1}, \dots, c_{k}))}}\).
  Hence we have \(\mathcal{D}_{n} \cup \{ \skE(I_{x}l_{k + 1}(x, c_{1}, \dots, c_{k})) \} \models \mathcal{D}_{n + 1} \) and \(L(\mathcal{D}_{n + 1}) = L(\mathcal{D}_{n}) \cup \{ c_{k + 1} \}\).
  Therefore, we have \(k + 1 < n + 1\) and \(L(\mathcal{D}_{n + 1}) \subseteq L(\mathcal{D}_{0}) \cup \{ c_{1}, \dots, c_{k + 1}\}\) and
  \[
    \mathcal{D}_{0} \cup \skE\left( \{ I_{x}l_{i}(x, c_{1}, \dots, c_{i - 1}) \mid 0 < i \leq k + 1\} \right) \models \mathcal{D}_{n + 1}. \qedhere
  \]
\end{proof}
\begin{proposition}
  \label{pro:15}
  Let \(\mathcal{S}\) be a sound and refutationally complete saturation system and \(T\) a theory.
  If \(\mathcal{S} + \LiteralAINDR_{2}\) refutes \(\cnf(\skE(T))\), then the saturation system \(\mathcal{S} + \INDRPF{\Literal(L(\skE(T)))}\) refutes \(\cnf(\skE(T))\).
\end{proposition}
\begin{proof}
  Assume that \(\mathcal{S} + \LiteralAINDR_{2}\) refutes \(\cnf(\skE(T))\), then by Lemma~\ref{lem:7} we obtain \(k \in \mathbb{N}\) and literals \(l_{i}(x, c_{1}, \dots, c_{i - 1})\) such that \(L(l_{i}(x, \vec{y})) \subseteq L(\skE(T))\) \(i = 1, \dots, k\).
  Moreover we have
  \[
    \cnf(\skE(T)) \cup \skE(\{ I_{x}l_{i}(x,c_{1}, \dots, c_{i - 1}) \mid 0 < i \leq k \}) \models \Box.
  \]
  Now we start with the clause set \(\cnf(\skE(T))\) and repeatedly apply the induction rule \(\INDRPF{\Literal(L(\skE(T)))}\) to derive the clause sets
  \[
    \mathcal{D}_{m} = \cnf(\skE(T)) \cup \cnf(\skE(\{ I_{x}l_{i}(x,c_{1}, \dots, c_{i - 1}) \mid 0 < i \leq m \})),
  \]
  for \(m = 1, \dots, k\).
  As shown above, this clause set \(\mathcal{D}_k\) is inconsistent, hence, by the refutational completeness of \(\mathcal{S}\) we obtain a refutation \(\mathcal{E}_{0}, \dots, \mathcal{E}_{n}\) from \(\mathcal{D}_{k}\).
  Hence the sequence \((\mathcal{D}_{1}, \dots, D_{k}, \mathcal{E}_{1}, \dots, \mathcal{E}_{n})\) is a \(\mathcal{S} + \INDRPF{\Literal(\skE(T))}\) refutation of \(\cnf(\skE(T))\).
\end{proof}
As an immediate consequence, we can transfer the previously established unprovability results to the concrete method described in \cite{reger2019,hajdu2020}.
\begin{theorem}
  \label{thm:2}
  Let \(\mathcal{S}\) be a sound and refutationally complete saturation system, then the system \(\mathcal{S} + \LiteralAINDR_{2}\) does neither refute the clause set \(\mathcal{X}_{m}\) nor the clause set \(\mathcal{Y}_{m,n}\) for \(0 < n < m\).
\end{theorem}
\begin{proof}
  We consider the case for the clause set \(\mathcal{X}_{m}\) with \(1 < m\).
  The other case is analogous.
  Suppose that \(\mathcal{S} + \LiteralAINDR_{2}\) refutes \(\mathcal{X}_{m}\), then by Proposition~\ref{pro:15} the saturation system \(\mathcal{S} + \INDRPF{\Literal(L(\mathcal{X}_{m}))}\) refutes \(\mathcal{X}_{m}\).
  This contradicts Theorem~\ref{thm:3}.
\end{proof}
Theorem~\ref{thm:2} gives us a family of simple and practically relevant clause sets that cannot be proved by the calculi presented in \cite{reger2019,hajdu2020}.

Let us now briefly discuss these results.
A possible source of criticism for Theorem~\ref{thm:2} may be that the underlying independence results (Lemma~\ref{lem:11}) are overly strong.
That is they do not exploit the restriction of the induction to literals, but instead rely on the fact that the sentences \(C_{m}\) and \(D_{m,n}\) with \(0 < n < m\) are already unprovable with induction for all quantifier-free formulas.
We can address this point by the following results.
\begin{lemma}
  \label{lem:21}
  The theory \(\mathcal{T} + \IND{\Literal(\mathcal{T})}\) proves \ref{B:1}--\ref{B:4}.
\end{lemma}
\begin{proof}
  Proving~\ref{B:2} and~\ref{B:3} is straightforward.
  For~\ref{B:4} we show the contrapositive \(y \neq z \rightarrow x + y \neq x + z\).
  We assume \(y \neq z\) and proceed by induction on \(x\) in the formula \(x + y \neq x + z\).
  For the base case we have to show \(0 + y \neq 0 + z\).
  By~\ref{B:2} and the definition of \(+\) the formula \(0 + y \neq 0 + z\) is equivalent to \(y \neq z\) which we have assumed.
  For the induction step we assume \(s(x) + y \neq s(x) + z\).
  By~\ref{B:2} and~\ref{ax:A:s} we obtain \(s(x + y) \neq s(x + z)\), hence \(x + y \neq x + z\) and we are done.

  Proving~\ref{B:1} is slightly more complicated because the induction interacts even more with the context.
  We assume \(x \neq 0\) and we have to show \(x = s(p(x))\).
  We proceed by induction on \(y\) in the formula \(x \neq y\).
  The induction base is trivial since we have assumed \(x \neq 0\).
  For the induction step we assume \(x \neq y_{0}\) and we have to show \(x \neq s(y_{0})\).
  Hence we assume \(x = s(y_{0})\).
  Now we have \(s(p(x)) = s(p(s(y_{0}))) = s(y_{0}) = x\) and we are done.
  Therefore we obtain the formula \(\Forall{y}{x \neq y}\) and in particular \(x \neq x\), which is a contradiction.
  Hence we obtain \(x = s(p(x))\).
\end{proof}
In the light of Shoenfield's theorem it is now clear that induction for literals is as powerful as quantifier-free induction.
\begin{proposition}
  \label{pro:17}
  \(\mathcal{T} + \IND{\Literal(\mathcal{T})} \equiv \mathcal{T} + \IND{\Open(\mathcal{T})}\).
\end{proposition}
\begin{proof}
  The direction from right to left is obvious.
  For the direction from left to right follows from Lemma~\ref{lem:21} and Shoenfield's Theorem (Theorem~\ref{thm:9}).
\end{proof}
The underlying independence results are therefore not too strong and it is not possible to improve the result by taking into account the restriction of the induction to literals.
The result may also be interesting from a practical point of view, because induction for literals is much easier to implement efficiently than induction for quantifier-free formulas.
It would therefore be interesting to investigate under which conditions induction for quantifier-free formulas collapses to induction for literals. 
However, we believe that there are practically relevant theories in which the induction schema for literals is strictly weaker than the induction schema for quantifier-free formulas.
Such a theory could allow us to provide unprovability results that give a motivation for the development of stronger induction mechanisms.

Another possible source of criticism is that our results focus on abstractions that are quite far from practical reality.
Most importantly, we do not exploit the fact that the induction rules \(\LiteralAINDR_{i}\) (\(i = 1, 2\)) attempt induction only for literals of which an instance of the dual literal occurs in the derived clauses.
Selecting the induction literals in this way seems to be a strong theoretical and practical restriction.
However, this restriction is crucial for current practical systems because it permits an efficient operation of the prover.
In practice, the restriction is usually weakened by the usage of heuristics for the selection of induction formulas \cite{hajdu2020}.
Another promising method for discovering induction formulas is introduced in \cite{claessen2013a,valbuena2015}, but it is unclear how to integrate this efficiently into a saturation-based system.
We currently do not have a candidate clause set that exploits the way in which \(\LiteralAINDR_{i}\) (\(i = 1, 2\)) select induction literals, but we plan to investigate this restriction in the future.

On the other hand, working with high-level abstractions allows us to obtain results that are robust against minor refinements of the induction rule from \cite{reger2019} such as the refinement proposed in \cite{hajdu2020}.
Moreover, the underlying independence results together with Lemmas~\ref{lem:3} and \ref{lem:2} suggest natural, yet not necessarily practical, extensions of the induction rule by allowing simultaneous induction on multiple variables or by allowing quantification inside the induction formula.

\section{Conclusion, Future Work, and Remarks}
\label{sec:future-work}
In this article we have analyzed a commonly used design principle for the integration of induction into saturation systems that has recently received increased interest \cite{kersani2013,kersani2014}, \cite{cruanes2015,cruanes2017}, \cite{wand2017}, \cite{echenim2020}, \cite{reger2019,hajdu2020}.

In Section~\ref{sec:unrestricted_induction_and_skolemization}, we have considered a general framework for induction over natural numbers in saturation-based provers that extend the language by Skolem symbols.
By reducing this induction mechanism to a logical theory (see Theorem~\ref{thm:4}), we have shown that in many relevant cases extending the language of the induction schema by Skolem symbols does not grant any additional power (see Proposition~\ref{pro:12}).
Furthermore, we have considered, in Section~\ref{sec:restricted_induction_and_skolemization}, an induction rule that restricts occurrences of Skolem symbols to ground terms according to similar restrictions observed in practical systems.
We have shown that under this restriction Skolem symbols correspond to induction parameters (see Theorem~\ref{thm:1}).
Finally, in Section~\ref{sec:unprovability}, we have used the results from Section~\ref{sec:restricted_induction_and_skolemization} and  independence results from the literature on mathematical logic to obtain some practically relevant unprovability results for the systems described in \cite{reger2019,hajdu2020} (see Theorem~\ref{thm:2}).

We plan to continue the work on induction in saturation-based theorem proving by analyzing the methods
developed by Cruanes \cite{cruanes2015,cruanes2017}, Wand \cite{wand2017} and Echenim and Peltier \cite{echenim2020}.
We are particularly interested in Cruanes' method because its mode of operation is very similar to the methods described in \cite{reger2019,hajdu2020}.
We suspect that under reasonable assumptions, the induction in Cruanes' system corresponds to the restricted induction rule (see Definition~\ref{def:8}) over \(\forall_{1}\) formulas.
Furthermore, Cruanes' method also allows induction on several formulas simultaneously and introduces definitions by the AVATAR splitting mechanism \cite{voronkov2014}.

Furthermore the work in this article has given rise to a number of questions that we hope to address in the future.
In Section~\ref{sec:unrestricted_induction_and_skolemization} we have established some very coarse results concerning the conservativity of extensions of the language of the induction formulas by Skolem symbols.
In particular we have shown that in many relevant cases extending the induction schema by Skolem symbols does not result in a more powerful system.
We have however left open the general case (see Question~\ref{que:3}).
This question is not proper to induction but is part of a more general question concerning the extension of the language of an axiom schema by Skolem symbols.
In Section~\ref{sec:restricted_induction_and_skolemization} we have mainly considered the case where the occurrences of all Skolem symbols in the induction formulas are subject to the restriction mentioned above.
Practical systems only impose this restriction on Skolem symbols that are generated by the induction rule.
We have left open the question about a characterization of these systems (see Question~\ref{que:1}).
Finally, it seems worthwhile to investigate the effects of the analyticity properties of induction rules used in concrete systems such as \cite{reger2019,hajdu2020} and their interaction with redundancy rules.
\bibliography{bibliography}
\bibliographystyle{alpha}

\appendix

\section{Appendix}
\label{sec:appendix:a}
This section provides a proof of Proposition~\ref{pro:2} and all the related lemmas.
The proof essentially proceeds by replacing each occurrence of a Skolem symbol by its definition.
We start by showing that we can isolate the occurrences of a given function symbol.
\begin{lemma}
  \label{lem:47}
  Let \(\varphi(\vec{z})\) be a formula and \(f\) a function symbol. Then there exists a formula \(\psi(\vec{z})\) such that \(\vdash \varphi(\vec{z}) \leftrightarrow \psi(\vec{z})\) and \(f\) occurs in \(\psi\) only in subformulas of the form \(x = f(\vec{t})\), where \(\vec{t}\) is free of \(f\).
\end{lemma}
\begin{proof}
  We exhaustively apply the equivalence from left to right
  \[
    \vdash \varphi(\vec{t}(u))) \leftrightarrow \Forall{x}{(x = u \rightarrow \varphi(\vec{t}(x)))},
  \]
  where \(x\) does not occur freely in \(\varphi(\vec{t}(u))\).
  It is straightforward to see that the logical equivalence of the formula so obtained is preserved and furthermore this transformation terminates because the overall nesting depth of \(f\) strictly decreases.
\end{proof}
After isolating a function symbol and assuming that it has a definition we can simply replace all the occurrences by its definition.
\begin{lemma}
  \label{lem:46}
  Let \(M\) be an \(L\) structure, \(\psi(\vec{x}, y)\) be an \(L\) formula such that \(M \models \ExistsExOne{y}{\psi(\vec{x}, y)}\).
  Let furthermore \(f\) be a function symbol and \(\varphi(\vec{z})\) an \(L \cup \{ f \}\) formula.
  Let \(M' \coloneqq M \cup \{ f \mapsto f_{\psi}\}\), where \(f_{\psi}(\vec{a}) = b\) with \(b \in |M|\) the only choice so that \(M \models \psi(\vec{a}, b)\).
  Then there exists an \(L\) formula \(\theta(\vec{z})\) such that \(M' \models \varphi(\vec{z}) \leftrightarrow \theta(\vec{z})\).
\end{lemma}
\begin{proof}
  This is easily seen by first observing that \(M' \models f(\vec{x}) = y \leftrightarrow \psi(\vec{x}, y)\).
  Now apply Lemma~\ref{lem:47} to \(\varphi\) in order to obtain a formula in which \(f\) occurs only as subformulas of the form \(y = f(\vec{t})\) with \(\vec{t}\) free of \(f\) and replace these occurrences with \(\psi(\vec{t}, y)\).
  Clearly the resulting formula is equivalent to \(\varphi\) in \(M'\).
\end{proof}
The assumption that a model has definable Skolem functions only provides us with definitions for Skolem symbols of \(L\) formulas.
The definitions for other Skolem symbols that are introduced at later stages need to be constructed based on the definitions for symbols of lower stages.
\begin{lemma}
  \label{lem:50}
  Let \(L\) be a Skolem-free first-order language and \(M\) an \(L\) structure with definable Skolem functions.
  Then there exists an expansion \(M'\) of \(M\) to \(\sk^{\omega}(L)\) such that \(M' \models \SASchema{L}\) and for each Skolem symbol \(f\) of \(\sk^{\omega}(L)\) then \(f^{M'}\) is \(L\) definable in \(M'\).
\end{lemma}
\begin{proof}
  We show by induction on \(i \in \mathbb{N}\) that there is an expansion \(M_{i}\) of \(M\) to the language \(\sk^{i}(L)\) such that for each Skolem symbol \(f/m\) of \(\sk^{i}(L)\) there exists a formula \(\psi_{f}(\vec{x}, y)\) such that \(M_{i} \models f(x_{1}, \dots, x_{m}) = y \leftrightarrow \psi_{f}(x_{1}, \dots, x_{m},y)\).
  The base case with \(i = 0\) is trivial.
  For the induction step we assume the claim for \(i\) and consider the case for \(i + 1\).
  Let \(f \coloneqq \mathfrak{s}_{\Quantifier{Q}{y}{\varphi(y,\vec{x})}}\) be a Skolem symbol of \(\sk^{i+1}(L)\), that does not belong to \(\sk^{i}(L)\).
  Let \(g_{1}/k_{1}\), \dots, \(g_{n}/k_{n}\) be the Skolem symbols occurring in the formula \(\varphi\).
  Then clearly \(g_{j}\) belongs to \(\sk^{i}(L)\) for all \(j = 1, \dots, n\).
  By the induction hypothesis there exist \(L\) formulas \(\psi_{g_{j}}(\vec{x}_{j})\) such that \(M_{i} \models g_{j}(x_{1}, \dots, x_{k_{j}}) = y \leftrightarrow \psi_{g_{i}}(x_{1}, \dots, x_{k_{j}}, y)\), for \(j = 1, \dots, n\).
  Then by repeated application of Lemma~\ref{lem:46} to the formula \(\varphi\), there exists an \(L\) formula \(\psi_{f}(\vec{x}, y)\) such that \(M_{i} \models \varphi(\vec{x}, y) \leftrightarrow \psi_{f}(\vec{x}, y)\).
  Since \(\psi_{f}\) is an \(L\) formula, \(M\) has definable Skolem functions, there exists a function \(h: |M|^{k} \to |M|\) and an \(L\) formula \(\delta_{h}(\vec{x}, y)\) such that \(h\) is defined in \(M\) by \(\delta_{h}\) and \(M \models \Exists{y}{\psi_{f}(\vec{x}, y) \rightarrow \psi_{f}(\vec{x}, h(\vec{x}))}\).
  We set \(f^{M_{i+1}} \coloneqq h\), then we have \(M_{i+1} \models f(\vec{x}) = y \leftrightarrow \delta_{h}(\vec{x}, y)\).
  It remains to show that \(M_{i+1}\) satisfies the Skolem axiom for \(f\).
  Suppose we have \(M_{i} \models \Exists{y}{\varphi(\vec{d}, y)}\), then we have \(M_{i} \models \Exists{y}{\psi_{f}(\vec{d}, y)}\).
  Hence \(M_{i + 1} \models \psi_{f}(\vec{d}, h(\vec{d}))\) and therefore \(M_{i + 1} \models \varphi(\vec{d}, f(\vec{d}))\).
  Hence \(M_{i + 1} \models \Exists{y}{\varphi(\vec{x}, y) \rightarrow \varphi(\vec{x}, f(\vec{x}))}\).
  Finally, we obtain \(M'\) by \(M' \coloneqq \bigcup_{i \geq 0} M_{i}\).
\end{proof}
Proving Proposition~\ref{pro:2} is now just a matter of replacing Skolem symbols in induction formulas by their definitions.
\begin{proof}[Proof of Proposition~\ref{pro:2}]
  Let \(\varphi\) be an \(L\) formula such that \(T + \SASchema{L} + \IND{\sk^{\omega}(L)} \vdash \varphi\).
  We proceed indirectly and assume \(T + \IND{L} \not \vdash \varphi\).
  Then there exists a model \(M\) of \(T + \IND{L}\) such that \(M \not \models \varphi\).
  By Lemma~\ref{lem:50} there exists an expansion \(M'\) of \(M\) to \(\sk^{\omega}(L)\) such that \(M' \models \SASchema{L}\) and for every Skolem symbol \(f\) there exists an \(L\) formula \(\delta_{f}(\vec{x}, y)\) such that \(M' \models f(\vec{x}) = y \leftrightarrow \delta_{f}(\vec{x}, y)\).
  Let \(\theta(x, \vec{z})\) be an \(\sk^{\omega}(L)\) formula and consider the induction axiom \(I_{x}\theta(x,\vec{z})\).
  By Lemma~\ref{lem:46} there exists an \(L\) formula \(\theta'(x, \vec{z})\) such that \(M' \models \theta(x, \vec{z}) \leftrightarrow \theta'(x, \vec{z})\).
  Hence we have \(M' \models I_{x}\theta(x, \vec{z}) \leftrightarrow I_{x}\theta'(x, \vec{z})\).
  Since \(M \models \IND{L}\), we have \(M' \models I_{x}\theta(x, \vec{z})\).
  Hence \(M' \models \IND{\sk^{\omega}(L)}\) and therefore \(M' \models T + \SASchema{L} + \IND{\sk^{\omega}(L)}\) but \(M' \not \models \varphi\).
  Contradiction!
\end{proof}

\section{Appendix}
\label{sec:appendix:b}
In this section we provide a model theoretic proof of Lemma~\ref{lem:13}.
The difficulty consists in showing that a given structure satisfies the induction schema \(\INDParameterFree{\Open(L(\mathcal{T}))}\).
In order to address this problem we start by simplifying the language of the induction schema (see Proposition~\ref{pro:8}).
By \(L'\) we denote the language \(L(\mathcal{T})\) without \(p\).
\begin{lemma}
  \label{lem:35}
  The theory \(\mathcal{T} + \INDParameterFree{\Open(L(\mathcal{T}))}\) proves the following formulas
  \begin{gather}
   x = 0 \vee x = s(p(x)). \label{lem:35:1} \\
   \text{\(x + \numeral{k} = \numeral{k} + x\)} \label{lem:35:2} \\
   \text{\(x + \numeral{k + 1} \neq x\)}, \label{lem:35:3}
 \end{gather}
 for all \(k \in \mathbb{N}\).
\end{lemma}
Next we show that whenever a \(p\)-free term contains a free variable \(x\), then whenever the variable \(x\) is substituted for \(s(x)\), we can propagate one occurrence of the successor function to the root of the term.
\begin{lemma}
  \label{lem:38}
  Let \(t(x)\) be a non-ground \(p\)-free term, then there exists a \(p\)-free term \(t'(x)\) such that \(\mathcal{T} \vdash t(s(x)) = s(t'(x))\).
\end{lemma}
\begin{proof}
  We proceed by induction on the structure of the term \(t\).
  If \(t = x\), then we are done by letting \(t' = t\).
  If \(t = s(u(x))\), then \(u\) is non-ground and \(p\)-free.
  We let \(t' = u(s(x))\), then we have \(\mathcal{T} \vdash t(s(x)) = s(u(s(x))) = s(t'(x))\).
  If \(t = u_{1} + u_{2}\), then we have to consider two cases depending on whether \(u_{2}\) is ground.
  If \(u_{2}\) is not ground, then by the induction hypothesis there exists \(u_{2}'\) such that \(\mathcal{T} \vdash u_{2}(s(x)) = s(u_{2}'(x))\).
  Then we have \(\mathcal{T} \vdash u_{1}(s(x)) + u_{2}(s(x)) = u_{1}(s(x)) + s(u_{2}'(x)) = s(u_{1}(s(x)) + u_{2}'(x)\) and we set \(t' = u_{1}(s(x)) + u_{2}'\).
  If \(u_{2}\) is ground, then \(u_{1}\) is non-ground and by the induction hypothesis there exists \(u_{1}'\) such that \(\mathcal{T} \vdash u_{1}(s(x)) = s(u_{1}'(x))\).
  We have \(\mathcal{T} \vdash t(s(x)) = u_{1}(s(x)) + u_{2} = s(u_{1}'(x)) + \numeral{k} = s(s^{k}(u_{1}'(x)))\), hence we choose \(t' = s^{k}(u_{1}')\).
\end{proof}
Now we will show that given a term \(t(x)\), we can eliminate the occurrences of \(p\) in \(t(s^{N}(x))\) when \(N\) is large enough.
\begin{lemma}
  \label{lem:37}
  Let \(t(x)\) be a term, then there exists \(N \in \mathbb{N}\) and a \(p\)-free term \(t\) such that \(\mathcal{T} \vdash t(s^{N}(x)) = t'\).
\end{lemma}
\begin{proof}
  If \(t\) is a ground term, then we have \(\mathcal{T} \vdash t = \numeral{k}\) for some \(k\) and we let \(t' = \numeral{k}\) and \(N = 0\).
  If \(t = x\), then we let \(N = 0\) and \(t = t'\).
  If \(t = s(u)\), where \(u\) is a term, then by the induction hypothesis there exists \(N'\) and a \(p\)-free \(u'\) such that \(\mathcal{T} \vdash u(s^{N'}(x)) = u'\).
  Hence we have \(\mathcal{T} \vdash t(s^{N}(x)) = s(u(s^{N}(x))) = s(u')\).
  Thus we let \(N \coloneqq N'\) and \(t' = s(u')\).
  If \(t = p(u)\), then by the induction hypothesis we have some \(N'\) and a \(p\)-free \(u'\) such that \(\mathcal{T} \vdash u(s^{N'}(x)) = u'\).
  Hence by Lemma~\ref{lem:38} we have \(\mathcal{T} \vdash p(u(s^{N' + 1}(x))) = p(u'(s(x))) = p(s(u^{\prime\prime})) = u^{\prime\prime}\), for some \(p\)-free term \(u''\) and we let \(N \coloneqq N' + 1\) and \(t' = u^{\prime\prime}\).
  If \(t = u_{1} + u_{2}\), then by the induction hypothesis there exists for \(i \in \{ 1, 2 \}\) a natural number \(N_{i}\) and a \(p\)-free term \(u_{i}'\) such that \(\mathcal{T} \vdash u_{i}(s^{N_{i}}(x)) = u_{i}'\).
  Let \(N = \max \{ N_{1}, N_{2} \}\), then we have \(\mathcal{T} \vdash t(s^{N}(x)) = u_{1}(s^{N}(x)) + u_{2}(s^{N}(x)) = u_{1}'(s^{N - N_{1}}(x)) + u_{2}'(s^{N - N_{2}}(x))\), thus we let \(t' = u_{1}'(s^{N - N_{1}}(x)) + u_{2}'(s^{N - N_{2}}(x))\).
\end{proof}
\begin{lemma}
  \label{lem:19}
  Let \(\varphi(x)\) be a formula, then there exists \(N \in \mathbb{N}\) and a \(p\)-free formula \(\varphi'(x)\) such that \(\mathcal{T} \vdash \varphi(s^{N}(x)) \leftrightarrow \varphi'\).
\end{lemma}
\begin{proof}
  Let \(\theta\) an atom, then \(\theta\) is of the form \(t_{1} = t_{2}\), then apply Lemma~\ref{lem:37} twice in order to obtain \(N_{1}\), \(N_{2}\) and the \(p\)-free terms \(t_{1}'(x)\) and \(t_{2}'(x)\).
  Now let \(N \coloneqq \max \{ N_{1}, N_{2}\}\) and observe that \(\mathcal{T} \vdash \theta(s^{N}(x)) \leftrightarrow t_{1}'(s^{N - N_{1}}(x)) = t_{2}'(s^{N - N_{2}}(x))\).
  
  Let \(\theta_{1}(x), \dots, \theta_{n}(x)\) be all the atoms of \(\varphi\).
  Let \(i \in \{ 1, \dots, n\}\), then apply the argument above to \(\theta_{i}\) in order to obtain a natural number \(M_{i}\) and a \(p\)-free atom \(\theta_{i}'\) such that \(\mathcal{T} \vdash \theta(s^{M_{i}}(x)) \leftrightarrow \theta_{i}'\).
  Let \(M = \max \{ M_{i} \mid i = 1, \dots, n\}\) and obtain \(\varphi'\) by replacing in \(\varphi(s^{M}(x))\) every atom \(\theta_{i}(s^{M}(x))\) by \(\theta_{i}'(s^{M - M_{i}}(x))\).
  Clearly we have \(\mathcal{T} \vdash \varphi(s^{M}(x)) \leftrightarrow \varphi'\).
\end{proof}
We can now ``factor'' the symbols \(p\) out of the induction schema.
The idea is instead of starting the induction at \(0\) we start the induction at some \(N \in \mathbb{N}\) that is large enough, so that we can eliminate \(p\) according to the lemma above.
\begin{lemma}
  \label{lem:36}
  \(\mathcal{T} + \eqref{B:1} + \INDParameterFree{\Open(L')} \vdash \INDParameterFree{\Open(\mathcal{T})}\).
\end{lemma}
\begin{proof}
  Let \(\varphi(x)\) be an \(L(\mathcal{T})\) formula.
  We want to show \(I_{x}\varphi(x)\).
  By Lemma~\ref{lem:19} above we obtain an \(N \in \mathbb{N}\) and a \(p\)-free formula \(\psi\) such that \(\mathcal{T} \vdash \varphi(s^{N}(x)) \leftrightarrow \psi(x)\).
  Now we work in \(\mathcal{T} + \eqref{B:1} + \INDParameterFree{\Open(L')}\) and assume \(\varphi(0)\) and \(\varphi(x) \rightarrow \varphi(s(x))\) and we want to show \(\varphi(x)\).
  Hence by a \(N - 1\) fold application of Lemma \eqref{B:1} it suffices to show \(\varphi(0)\), \(\varphi(\numeral{1})\), \dots, \(\varphi(s^{N}p^{N}(x))\).
  By starting with \(\varphi(0)\) and iterating \(\varphi(x) \rightarrow \varphi(s(x))\) we obtain \(\varphi(\numeral{n})\) for all \(n \in \mathbb{N}\).
  Hence it remains to show \(\varphi(s^{N}(p^{N}(x)))\).
  We proceed by induction on \(\psi\).
  For the induction base we have to show \(\psi(0)\) which is equivalent to \(\varphi(\numeral{N})\), hence we are done.
  For the induction step we assume \(\psi(x)\) and we have to show \(\psi(s(x))\).
  We have \(\psi(x) \leftrightarrow \varphi(s^{N}(x))\) and by \(\Forall{x}{\left(\varphi(x) \rightarrow \varphi(s(x))\right)}\) we obtain \(\varphi(s^{N}(x)) \rightarrow \varphi(s^{N + 1}(x))\) thus by modus ponens \(\varphi(s^{N + 1}(x))\) which is equivalent to \(\psi(s(x))\).
  This completes the induction step.
  By the induction we thus obtain \(\psi(x)\), and in particular \(\psi(p^{N}(x))\) which is equivalent to \(\varphi(s^{N}p^{N}(x))\).
  This completes the proof.
\end{proof}
As an immediate consequence of the above lemma we can factor all the occurrences of \(p/1\) in the induction formulas into a single axiom.
\begin{proposition}
  \label{pro:8}
  \(\mathcal{T} + \INDParameterFree{\Open(L(\mathcal{T}))} \equiv \mathcal{T} + \eqref{B:1} + \INDParameterFree{\Open(L')}\).
\end{proposition}
Over \(\mathbb{Z}\) \(p\)-free atoms in one variable represent equations between two linear functions as can be easily seen.
Linear functions have the nice property that either they coincide everywhere or else they intersect in at most one point.
This property of linear functions allows us to analyze the truth values of a quantifier-free formula in one point. 
The idea is that this property allows us to define a radius, beyond which an atom behaves on the positive integers just like on the negative integers.
We will now define an \(L(\mathcal{T})\) structure that has an analogous property.
\begin{definition}
  \label{def:12}
  Let \(\mathcal{M}\) be the \(L(\mathcal{T})\) structure whose domain is the set of pairs \((b, n) \in \{0,1\} \times \mathbb{Z}\) such that \(b = 0\) implies \(n \in \mathbb{N}\) and that interprets the function symbols \(0\), \(s\), \(p\), and \(+\) as follows
  \begin{gather*}
    0^{\mathcal{M}} = (0, 0), \\
    s^{\mathcal{M}}((b,n)) = (b, n+1), \\
    p^{\mathcal{M}}((0, n)) = (0, n \, \dot{-} \, 1), \\
    p^{\mathcal{M}}((1, n)) = (1, n - 1), \\
    (b_{1}, n_{1}) +^{\mathcal{M}} (b_{2}, n_{2}) = (\max\{b_{1},b_{2}\}, n_{1} + n_{2}),
  \end{gather*}
  where \(b, b_{1}, b_{2} \in \{0,1\}\) and \(n, n_{1}, n_{2} \in \mathbb{N}\).
\end{definition}
It is clear that the structure \(\mathcal{M}\) is indeed an \(L(\mathcal{T})\) structure.
\begin{lemma}
  \label{lem:30}
  \(\mathcal{M} \models \mathcal{T} + \eqref{B:1}\).
\end{lemma}
\begin{proof}
  The element \(0^{\mathcal{M}}\) clearly has no predecessor.
  Furthermore \(p^{\mathcal{M}}0^{\mathcal{M}} = (0,0) = 0^{\mathcal{M}}\).
  We have \(p^{\mathcal{M}}(s^{\mathcal{M}}(b,n))) = (b,n)\).
  Moreover \((b,n) +^{\mathcal{M}} (0, 0) = (b, n)\) and
  \begin{multline*}
    (b_{1},n_{1}) +^{\mathcal{M}} s^{\mathcal{M}}((b_{2},n_{2})) = (b_{1}, n_{1}) +^{\mathcal{M}}(b_{2}, n_{2} + 1) \\ = (\max\{b_{1},b_{2}\}, n_{1} + n_{2} + 1) = s^{\mathcal{M}}((\max\{b_{1}, b_{2}\}, n_{1} + n_{2})) \\ = s^{\mathcal{M}}((b_{1},n_{1}) +^{\mathcal{M}} (b_{2}, n_{2})).
  \end{multline*}
  Finally, observe that every element that is not \(0^{\mathcal{M}}\) has a predecessor.
\end{proof}
\begin{lemma}
  \label{lem:29}
  Let \(t(x)\) be a \(p\)-free term containing the variable \(x\), then
  \begin{equation}
    \label{eq:1}
    t^{\mathcal{M}}((b,n)) = (b, t^{\mathbb{Z}}(n)).
  \end{equation}
\end{lemma}
\begin{proof}
  We proceed by induction on the structure of the term \(t\).
  If \(t\) is the variable \(x\), then \(t^{\mathcal{M}}(b,n) = (b,n) = (b, t^{\mathbb{Z}}(n))\).
  If \(t = t_{1} + t_{2}\), then either \(t_{1}\) or \(t_{2}\) is not ground.
  If \(t_{1}\) contains \(x\) and \(t_{2}\) does not contain \(x\), then we have by the induction hypothesis \(t_{1}^{\mathcal{M}}(b,n) = (b, t_{1}^{\mathbb{Z}}(n))\) and \(t_{2}^{\mathcal{M}}(b,n) = (0, t_{2}^{\mathbb{N}})\).
  Hence \(t^{\mathcal{M}}(b,n) = (b, t_{1}^{\mathbb{Z}}(n) + t_{2}^{\mathbb{N}}) = (b, t^{\mathbb{Z}}(n))\).
  If both \(t_{1}\) and \(t_{2}\) contain \(x\), then we have by the induction hypothesis \(t_{1}^{\mathcal{M}}(b,n) = (b, t_{1}^{\mathbb{Z}}(n))\) and \(t_{2}^{\mathcal{M}}(b,n) = (b, t_{2}^{\mathbb{Z}}(n))\).
  Hence \(t^{\mathcal{M}}(b,n) = (b, t_{1}^{\mathbb{Z}}(n) + t_{2}^{\mathbb{Z}}(n)) = (b, t^{\mathbb{Z}}(n))\).
\end{proof}
The following lemma expresses the informal idea discussed above that an atom is determined outside of some finite radius.
\begin{lemma}
  \label{lem:34}
  Let \(\theta(x)\) be a \(p\)-free atom, then there exists \(N \in \mathbb{N}\) such that for all \(n \geq N\)
  \begin{equation}
    \label{eq:43}
    \mathcal{M} \models \theta((1, -n)) \Longleftrightarrow \mathbb{N} \models \theta(n).
  \end{equation}
\end{lemma}
\begin{proof}
  Let \(\theta(x) \coloneqq (t_{1} = t_{2})\).
  If \(t_{1}\) and \(t_{2}\) do not contain \(x\), then the claim holds trivially.
  If \(t_{1}\) and \(t_{2}\) both contain \(x\), then by Lemma~\ref{lem:29} we have
  \begin{equation}
    \label{eq:44}
    \mathcal{M} \models \theta((1,-n)) \Leftrightarrow \mathbb{Z} \models \theta(-n),
  \end{equation}
  for all \(n \in \mathbb{N}\).
  In \(\mathbb{Z}\) the atom \(\theta\) is an equation between two linear functions.
  Hence there are two cases to consider.
  If \(\theta\) is true in \(\mathbb{Z}\) in at most one point, then there exists \(N \in \mathbb{N}\) such that for all \(m \in \mathbb{Z}\) with \(|m| \geq N\) we have \(\mathbb{Z} \not \models \theta(m)\).
  Thus \(\mathcal{M} \not \models \theta((1,-m))\) and \(\mathbb{N} \not \models \theta(m)\) for \(m \geq N\).
  Otherwise, if \(\theta\) is true in more than one point of \(\mathbb{Z}\), then \(\theta\) is true everywhere in \(\mathbb{Z}\) and we have \(\mathcal{M} \models \theta((1,-m))\) and \(\mathbb{N} \models \theta(m)\) for all \(m \geq 0\).
  If \(t_{1}\) contains \(x\), but \(t_{2}\) does not contain \(x\), then clearly \(\mathcal{M} \not \models \theta((1, m))\) for all \(m \in \mathbb{Z}\).
  Moreover \(\mathcal{M} \models \theta((0,m))\) if and only if \(\mathbb{N} \models \theta(m)\) for all \(m \geq 0\).
  Clearly \(\mathbb{N} \models \theta(m)\) for at most one \(m \in \mathbb{N}\), hence there exists \(N \in \mathbb{N}\) such that \(\mathbb{N} \not \models \theta(m)\) for all \(m \geq N\).
  Hence we have \(\mathcal{M} \not \models \theta((1, -m))\) and \(\mathcal{M} \not \models \theta((0, m))\) for all \(m \geq N\).
  This completes the proof.
\end{proof}
Thanks to the property shown in the lemma above, we can now quite easily show that \(\mathcal{M}\) is a model of \(\INDParameterFree{\Open(L')}\).
\begin{lemma}
  \label{lem:32}
  \(\mathcal{M} \models \INDParameterFree{\Open(L')}\).
\end{lemma}
\begin{proof}
  Let \(\varphi(x)\) be a quantifier-free \(p\)-free formula and assume that \(\mathcal{M} \models \varphi(0)\) and \(\mathcal{M} \models \varphi(x) \rightarrow \varphi(s(x))\).
  Let \((b, n) \in |\mathcal{M}|\).
  If \(b = 0\), then the claim follows by a straightforward induction and the definition of the model \(\mathcal{M}\).
  If \(b = 1\), then we consider the atoms of the formula \(\varphi(x)\).
  By applying Lemma~\ref{lem:34} to the atoms of \(\varphi\) we obtain a natural number \(M\) such that for all \(m \geq M\) we have \(\mathcal{M} \models \varphi((1,-m)) \Leftrightarrow \mathcal{M} \models \varphi((0,m))\).
  Clearly, there exists a natural number \(n'\) with \(n' \leq n\) and \(n' \leq -M\).
  Then we have \(\mathcal{M} \models \varphi((1, n'))\) because we have already shown that \(\mathcal{M} \models \varphi((0,-n'))\).
  By applying repeatedly applying the induction step we then obtain \(\mathcal{M} \models \varphi((1,n' + k))\) for all \(k \in \mathbb{N}\).
  In particular we obtain \(\mathcal{M} \models \varphi((1, n))\).
\end{proof}
We can now finally give a proof of Lemma~\ref{lem:13}
\begin{proof}[Proof of Lemma~\ref{lem:13}]

  By Lemma~\ref{lem:30} and Lemma~\ref{lem:32} we have \(\mathcal{M} \models \mathcal{T} + \eqref{B:1} + \INDParameterFree{\Open(L')}\).
  Now observe that \((1,0) +^{\mathcal{M}}(1,0) = (1,0)\) but \((1,0) \neq (0,0) = 0^{\mathcal{M}}\).
  Hence \(\mathcal{T} + \eqref{B:1} + \INDParameterFree{\Open(L')} \not \vdash \theta(x,x)\).
  Hence by Proposition~\ref{pro:8} we obtain \(\mathcal{T} + \INDParameterFree{\Open(L(\mathcal{T}))} \not \vdash \theta(x,x)\).
\end{proof}
\begin{acronym}
  \acro{AITP}{automated inductive theorem proving}
  \acro{ATP}{automated theorem proving}
\end{acronym}

\end{document}